\documentclass{amsart}
\usepackage{amscd,amsmath,amssymb,amsfonts}
\usepackage[cmtip, all]{xy}

\newtheorem{thm}{Theorem}[section]
\newtheorem{prop}[thm]{Proposition}
\newtheorem{lem}[thm]{Lemma}
\newtheorem{cor}[thm]{Corollary}
\newtheorem{conj}[thm]{Conjecture}

\renewcommand{\theclaim}{\kern-3pt}

\newtheorem{IntroThm}{Theorem}

\theoremstyle{definition}
\newtheorem{Def}[thm]{Definition}

\theoremstyle{remark}
\newtheorem{rem}[thm]{Remark}

\numberwithin{equation}{section}

\newcommand{\sE}{{\mathcal E}}
\newcommand{\sF}{{\mathcal F}}

\newcommand{\sH}{{\mathcal H}}

\newcommand{\sO}{{\mathcal O}}

\newcommand{\sQ}{{\mathcal Q}}

\newcommand{\sS}{{\mathcal S}}
\newcommand{\sT}{{\mathcal T}}

\newcommand{\sZ}{{\mathcal Z}}
\newcommand{\A}{{\mathbb A}}

\newcommand{\G}{{\mathbb G}}

\renewcommand{\P}{{\mathbb P}}

\newcommand{\mS}{{\mathbb S}}

\newcommand{\Z}{{\mathbb Z}}

\renewcommand{\phi}{\varphi}

\newcommand{\red}{{\rm red}}

\newcommand{\Hom}{{\rm Hom}}

\newcommand{\Spec}{\operatorname{Spec}}

\newcommand{\Char}{\operatorname{char}}
\newcommand{\Tr}{{\rm Tr}}

\newcommand{\0}{\emptyset}
\newcommand{\sHom}{{\mathcal{H}{om}}}

\newcommand{\id}{{\operatorname{id}}}
\newcommand{\Zar}{{\text{\rm Zar}}}

\newcommand{\holim}{\mathop{{\rm holim}}}
\newcommand{\op}{{\text{\rm op}}}

\newcommand{\fib}{{\operatorname{\rm fib}}}

\newcommand{\Spt}{{\mathbf{Spt}}}
\newcommand{\Spc}{{\mathbf{Spc}}}
\newcommand{\Sm}{{\mathbf{Sm}}}

\renewcommand{\lim}{\operatornamewithlimits{\varprojlim}}
\newcommand{\colim}{\operatornamewithlimits{\varinjlim}}

\newcommand{\GL}{{\operatorname{\rm GL}}}

 \newcommand{\Ab}{{\mathbf{Ab}}}

\newcommand{\HZ}{{\operatorname{\sH \Z}}}

\newcommand{\SH}{{\operatorname{\sS\sH}}}

\newcommand{\eff}{{\mathop{eff}}}

\newcommand{\DM}{{DM}}

\newcommand{\Nis}{{\operatorname{Nis}}}

\newcommand{\ds}{{/\kern-3pt/}}

\newcommand{\EM}{{{EM}_{\A^1}}}
\newcommand{\eM}{{{EM}}}
\newcommand{\SmCor}{{\mathop{SmCor}}}

\newcommand{\mot}{{\mathop{mot}}}
\newcommand{\Mot}{{\mathop{Mot}}}

\newcommand{\cone}{{\operatorname{cone}}}

\newcommand{\co}{{co\text{-}tr}}

\newcommand{\Iso}{\text{Iso}}
\newcommand{\sFS}{\mathcal{FS}}

\begin{document}

\title{Slices and transfers}
\author{Marc Levine}
\address{Northeastern University\\
Department of Mathematics\\
Boston, MA 02115\\
U.S.A.}
\address{Universit\"at Duisburg-Essen\\
Fakult\"at Mathematik, Campus Essen\\
45117 Essen\\
Germany}
\email{marc.levine@uni-due.de}
\thanks{Research supported by the NSF grant DMS-0801220 and the Alexander von Humboldt Foundation}


 
\renewcommand{\abstractname}{Abstract}
\begin{abstract}  We study the slice filtration for $S^1$-spectra, and raise a number of questions regardings its properties.
\end{abstract}
\date{\today}
\maketitle
\tableofcontents

\section*{Introduction} 
Voevodsky \cite{VoevSlice} has defined an analog of  the classical Postnikov tower in the
setting of  motivic stable homotopy theory by replacing the simplicial suspension
$\Sigma_s:=-\wedge S^1$ with $\P^1$-suspension
$\Sigma_{\P^1}:=-\wedge\P^1$; we call this construction the {\em motivic Postnikov tower}.

Let $\SH(k)$ denote the motivic stable homotopy category of $\P^1$-spectra. One of the main results on motivic Postnikov tower in this setting is

\begin{IntroThm} For $E\in\SH(k)$, the slices $s_nE$ have the natural structure of an $\HZ$-module, and hence determine objects in the category of motives $\DM(k)$.
\end{IntroThm}

The statement is a bit imprecise, as the following expansion will make clear: Ostvar-R\"ondigs \cite{OstRond, OstRond2}  have shown that the homotopy category of strict $\HZ$-modules is equivalent to the category of motives $\DM(k)$. Additionally, Voevodsky \cite{VoevSlice} and the author \cite{LevineHC} have shown that the 0th slice of the sphere spectrum $\mS$ in $\SH(k)$ is isomorphic to $\HZ$. Thus, for $E\in\SH(k)$, the canonical $\mS$-module structure on $E$ induces an $\HZ$-module structure on the slices $s_nE$, in $\SH_{S^1}(k)$. This  has been refined to the model category level by Pelaez \cite{Pelaez}, showing that the slices of a $\P^1$-spectrum $E$ have a natural structure of a strict $\HZ$-module, hence are motives. 

Let $\Spt_{S^1}(k)$ denote the category of $S^1$-spectra. The motivic analog is the category of effective motives over $k$,  $\DM^\eff(k)$. We consider the motivic Postnikov tower in the homotopy category of $S^1$-spectra, $\SH_{S^1}(k)$, and ask  the  question:
\begin{enumerate}
\item is there a ring object in $\Spt_{S^1}(k)$, $\HZ^\eff$, such that the homotopy category of $\HZ^\eff$ modules is equivalent to the category of effective motives $\DM^\eff(k)$?
\item What properties (if any) need an $S^1$-spectrum $E$ have so that the slices $s_nE$ have a natural structure ofhave a natural structure as the Eilenberg-Maclane spectrum of a homotopy invariant complex of presheaves with transfer?
\end{enumerate}
Naturally, if $\HZ^\eff$ exists as in (1), we are asking the slices in (2) to be (strict) $\HZ^\eff$ modules. Of course, a natural candidate for $\HZ^\eff$ would be the 0-$S^1$-spectrum of $\HZ$, $\Omega^\infty_{\P^1}\HZ$, but as far as I know, this property has not yet been investigated.

As we shall see, the 0-$S^1$-spectrum of a $\P^1$-spectrum does have the property that its ($S^1$) slices are motives, while one can give examples of $S^1$-spectra for which the 0th slice does not have this property. This suggests a relation of the question of the structure of the slices of an $S^1$-spectrum with a motivic version of the recognition problem: 
\begin{enumerate}
\item[(3)] How can one tell if a given $S^1$-spectrum is an $n$-fold $\P^1$-loop spectrum? 
\end{enumerate}

In this paper, we prove two main results about the ``motivic" structure on the slices of $S^1$-spectra:
\begin{IntroThm}\label{IntroThm:Tower} Suppose $\Char k=0$.   Let $E$ be an $S^1$-spectrum. Then for each $n\ge1$, there is a tower
\[
\ldots\to\rho_{\ge p+1}s_nE\to \rho_{\ge p}s_nE\to \ldots\to s_nE
\]
in $\SH_{S^1}(k)$ with the following properties: 
\begin{enumerate}
\item the tower is natural in $E$.
\item Let  $\bar{s}_{p,n}E$ be the cofiber of $\rho_{\ge p+1}s_nE\to \rho_{\ge p}s_nE$. Then there is a homotopy invariant complex of presheaves with transfers $\hat{\pi}_p((s_nE)^{(n)})^*\in \DM^\eff_-(k)$ and a natural isomorphism in $\SH_{S^1}(k)$, 
\[
\EM(\hat{\pi}_p((s_nG)^{(n)})^*)\cong \bar{s}_{p,n}E,
\]
where $\EM:\DM^\eff_-(k)\to \SH_{S^1}(k)$ is the   Eilenberg-Maclane spectrum functor.
\end{enumerate}
\end{IntroThm}
This result is proven in section \ref{sec:GenCyc}.

One can say a bit more about the tower appearing in theorem~\ref{IntroThm:Tower}. For instance,
$\holim_p\fib( \rho_{\ge p}s_nE\to s_nE)$ is weakly equivalent to zero, so the spectral sequence associated to this tower is weakly convergent. If $s_nE$ is globally $N$-connected (i.e., there is an $N$ such that $s_nE(X)$ is $N$-connected for all $X\in\Sm/k$) then the spectral sequence is strongly convergent.

In other words, the {\em higher} slices of an arbitrary $S^1$-spectrum have some sort of transfers  ``up to filtration". The situation for the 0th slice appears to be more complicated, but for a $\P^1$-loop spectrum we have at least the following result:

\begin{IntroThm}\label{IntroThm:Loops}  Suppose $\Char k=0$.  Take $E\in \SH_{\P^1}(k)$. Then 
for all $m$,  the homotopy sheaf $\pi_m(s_0\Omega_{\P^1}E)$ has a natural structure of a homotopy invariant sheaf with transfers.
\end{IntroThm}

We actually prove a more precise result (corollary~\ref{cor:Trans0}) which states that the 0th slice $s_0\Omega_{\P^1}E$ is itself  a presheaf with transfers, with values in the stable homotopy category $\SH$, i.e., $s_0\Omega_{\P^1}E$ has ``transfers up to homotopy". This raises the question:
\begin{enumerate}
\item[(4)] Is there an operad acting on $s_0\Omega^n_{\P^1}E$ which shows that  $s_0\Omega^n_{\P^1}E$ admits transfers up to homotopy and higher homotopies up to some level?
\end{enumerate}

Part of the motivation for this paper came out of discussions with H\'el\`ene Esnault concerning the (admittedly vague) question: Given a smooth projective variety $X$ over some field $k$, that admits a 0-cycle of degree 1, are there ``motivic" properties of $X$ that lead to the existence of a $k$-point?  The fact that the existence of 0-cycles of degree 1 has something to do with the transfer maps from 0-cycles on $X_L$ to zero-cycles on $X$, as $L$ runs over finite field extensions of $k$, while the lack of a transfer map in general appears to be closely related to the subtlety of the existence of $k$-points led to our inquiry into the ``motivic" nature of the spaces $\Omega_{\P^1}^n\Sigma_{\P^1}^nX_+$, or rather, their associated $S^1$-spectra.

\ \\
{\bf Notation and conventions}.
In this paper, we will be passing from the unstable motivic (pointed) homotopy category over $k$,  $\sH_\bullet(k)$, to the homotopy category of motivic $S^1$-spectra over $k$, $\SH_{S^1}(k)$, via the infinite (simplicial) suspension functor
\[
\Sigma_s^\infty:\sH_\bullet(k)\to \SH_{S^1}(k)
\]
For a smooth $k$-scheme $X\in \Sm/k$ and a subscheme $Y$ of $X$ (sometimes closed, sometimes open), we let $(X,Y)$ denote the homotopy push-out in the diagram
\[
\xymatrix{
Y\ar[r]\ar[d]&X\\
\Spec k}
\]
and as usual write $X_+$ for $(X\amalg\Spec k,\Spec k)$. We often denote $\Spec k$ by $*$. For an object $S$ of $\sH_\bullet(k)$, we often use $S$ to denote $\Sigma_s^\infty S\in \SH_{S^1}(k)$ when the context makes the meaning clear; we also use this convention when passing to various localizations of $\SH_{S^1}(k)$.

Regarding the categories $\Spt_{S^1}(k)$, $\SH_{S^1}(k)$ and $\SH(k)$, we will use the notation spelled out in \cite{LevineHC}. In addition to this source, we refer the reader to \cite{Jardine, MorelLec, MorelVoev, OstRond2, OstRond2, VoevSlice}. For details on the category $\DM^\eff(k)$, we refer the reader to \cite{CisDeg,FSV}.\\

\noindent{\em Dedication}. This paper is warmly dedicated to Andrei Suslin, who has provided me more inspiration than I can hope to tell.

\section{Infinite $\P^1$-loop spectra} We first consider the case of the 0-$S^1$-spectrum of a $\P^1$-spectrum. We let
\begin{align*}
&\Omega^\infty_{\P^1}:\SH(k)\to \SH_{S^1}(k)\\
&\Omega^\infty_{\P^1,\mot}:\DM(k)\to \DM^\eff(k)
\end{align*}
be the (derived) 0-spectrum (resp. 0-complex) functor,  let
\begin{align*}
&\EM:\DM(k)\to \SH(k)\\
&\EM^\eff:\DM^\eff(k)\to \SH_{S^1}(k)
\end{align*}
the respective Eilenberg-Maclane spectrum functors.

\begin{thm} Fix an integer $n\ge0$. Then there is a functor
\[
\Mot^\eff(s_n):\SH(k)\to \DM^\eff(k)
\]
and a natural isomorphism
\[
\phi_n:\EM^\eff\circ \Mot^\eff(s_n)\to s^\eff_n\circ\Omega^\infty_{\P^1}
\]
of functors from $\SH(k)$ to $\SH_{S^1}(k)$.In other words, for $\sE\in \SH(k)$, there is a canonical lifting of the slice $s^\eff_n(\Omega^\infty_{\P^1}\sE)$ to a motive $ \Mot^\eff(s_n)(\sE)$.
\end{thm}

\begin{proof} By Pelaez, there is a functor
\[
\Mot(s_n):\SH(k)\to \DM(k)
\]
and a natural isomorphism
\[
\Phi_n:\EM\circ\Mot(s_n)\to s_n
\]
i.e., the slice $s_n\sE$ lifts canonically to a motive $\Mot(s_n)(\sE)$. Now apply the 0-complex functor to define
\[
 \Mot^\eff(s_n):=\Omega^\infty_{{\P^1},\mot}\circ \Mot(s_n).
 \]
 We have canonical isomorphisms
 \begin{align*}
 \EM^\eff\circ \Omega^\infty_{{\P^1},\mot}\circ \Mot(s_n)&\cong \Omega^\infty_{\P^1}\circ \EM\circ  \Mot(s_n)\\
 \cong \Omega^\infty_{\P^1}\circ s_n\\
 \cong s_n^\eff\circ  \Omega^\infty_{\P^1}
 \end{align*}
 as desired.
 \end{proof}
 
 In other words, the slices of an infinite ${\P^1}$-loop spectrum are effective motives.
 
 \section{An example}\label{sec:Example} We now show that the 0th slice of an $S^1$-spectrum is not always a motive. In fact, we will give an example of an Eilenberg-Maclane spectrum whose $0$th slice does not have transfers.
 
 For this, note the following: 
 \begin{lem} \label{lem:Galois} Let $p:Y\to X$ be a finite Galois cover in $\Sm/k$, with Galois group $G$. Let $\sF$ be a presheaf with transfers on $\Sm/k$. Then the composition
 \[
 p^*\circ p_*:\sF(Y)\to \sF(Y)
 \]
 is given by
 \[
 p^*\circ p_*(x)=\sum_{g\in G}g^*(x)
 \]
 \end{lem}
 
 \begin{proof} Letting $\Gamma_p\subset Y\times X$ be the graph of $p$, and $\Gamma_g\subset Y\times Y$ the graph of $g:Y\to Y$ for $g\in G$,  one computes that
 \[
 \Gamma_p^t\circ \Gamma_p=\sum_{g\in G}\Gamma_g,
 \]
 whence the result.
 \end{proof}
 
 Now let $C$ be a smooth projective  curve over $k$, having no $k$-rational points. We assume that $C$ has genus $g>0$,  so every map $\A^1_F\to C_F$ over a field $F\supset k$ is constant ($C$ is {\em $\A^1$-rigid}).

Let $\Z_C$ be the representable presheaf:
 \[
 \Z_C(Y):=\Z[\Hom_{\Sm/k}(Y,C)]
 \]
$\Z_C$ is automatically a Nisnevich sheaf; since $C$ is $\A^1$-rigid, $\Z_C$ is also homotopy invariant. Furthermore $\Z_C$ is a {\em birational} sheaf, that is, for each dense open immersion $U\to Y$ in $\Sm/k$, the restriction map $\Z_C(Y)\to \Z_C(U)$ is an isomorphism. Indeed, it is the same to say that $\Hom_{\Sm/k}(Y,C)\to \Hom_{\Sm/k}(U,C)$ is an isomorphism. If now $f:U\to C$ is a morphism, then the closure $\bar{\Gamma}$ in $Y\times C$ of the graph of $f$ maps birationally to $Y$ via the projection. But since $Y$ is regular,  each  fiber of $\bar{\Gamma}\to Y$ is rationally connected, hence maps to a point of $C$, and thus $\bar{\Gamma}\to Y$ is birational and 1-1. By Zariski's main theorem, $\bar{\Gamma}\to Y$ is an isomorphism, hence $f$ extends to $\bar{f}:Y\to C$, as claimed.

Next, $\Z_C$ satisfies Nisnevich excision. This is just a general property of birational sheaves. In fact, let
\[
\xymatrix{
V\ar[r]^{j_V}\ar[d]_{f_{|V}}&Y\ar[d]^f\\
U\ar[r]_{j_U}&X}
\]
be an elementary Nisnevich square, i.e., the square is cartesian, $f$ is \'etale, $j_U$ and $j_V$ are open immersions, and $f$ induces an isomorphism $Y\setminus V\to X\setminus U$. We may assume that $U$ and $V$ are dense in $X$ and $Y$. Let $\sF$ be a birational sheaf on $\Sm/k$, and apply $\sF$ to this diagram. This gives us the square
\[
\xymatrix{
\sF(X)\ar[r]^{j_U^*}\ar[d]_{f^*}&\sF(U)\ar[d]^{f_{|V}^*}\\
\sF(Y)\ar[r]_{j_V^*}&\sF(V)}
\]
As the horizontal arrows are isomorphisms, this square is homotopy cartesian, as desired. 

In particular, the (simplicial) Eilenberg-Maclane spectrum $\eM_s(\Z_C)$ is weakly equivalent as a presheaf on $\Sm/k$ to its fibrant model in $\SH_{S^1}(k)$ ($\eM_s(\Z_C)$ is {\em quasi-fibrant}). In addition, the canonical map
\[
\eM_s(\Z_C)\to s_0(\eM_s(\Z_C))
\]
is an isomorphism in $\SH_{S^1}(k)$. Indeed, since $\eM_s(\Z_C)$ is quasi-fibrant, a quasi-fibrant model for $s_0(\eM_s(\Z_C))$ may be computed by using the method of \cite[\S 5]{LevineHC}  as follows: Take $Y\in \Sm/k$ and let $F=k(Y)$. Let $\Delta^n_{F,0}$ be the {\em semi-local} algebraic $n$-simplex, that is, $\Delta^n_{F,0}=\Spec(\sO_{\Delta^n_F,v})$, where $v=\{v_0,\ldots, v_n\}$ is the set of vertices in $\Delta^n_F$, and  $\sO_{\Delta^n_F,v}$ is the semi-local ring of $v$ in $\Delta^n_F$. The assignment $n\mapsto \Delta^n_{F,0}$ forms a cosimplicial subscheme of $n\mapsto \Delta^n_F$ and for a quasi-fibrant $S^1$-spectrum $E$, there is a natural isomorphism in $\SH$
\[
s_0(E)(Y)\cong E(\Delta^*_{F,0}),
\]
where $E(\Delta^*_{F,0})$ denotes the total spectrum of the simplicial spectrum $n\mapsto E(\Delta^n_{F,0})$. If now $E$ happens to be a birational $S^1$-spectrum, meaning that $j^*:E(Y)\to E(U)$ is a weak equivalence for each dense open immersion $j:U\to Y$ in $\Sm/k$, then the restriction map
\[
j^*:E(\Delta^*_Y)\to E(\Delta^*_{F,0})\cong s_0(E)(Y)
\]
is a weak equivalence. Thus, as $E$ is quasi-fibrant and hence homotopy invariant, we have the sequence of isomorphisms in $\SH$
\[
E(Y)\to E(\Delta^*_Y)\to E(\Delta^*_{F,0})\cong s_0(E)(Y),
\]
and hence $E\to s_0(E)$ is an isomorphism in $\SH_{S^1}(k)$. Taking $E=\eM_s(\Z_C)$ verifies our claim.

Finally, $\Z_C$ does not admit transfers. Indeed, suppose $\Z_C$ has transfers. Let $k\to L$ be a Galois extension such that $C(L)\neq\0$; let $G$ be the Galois group. Since $\Z_C(k)=\{0\}$ (as we have assumed that $C(k)=\0$), the push-forward map
\[
p_*:\Z_C(L)\to \Z_C(k)
\]
is the zero map, hence $p^*\circ p_*=0$. But for each $L$-point $x$ of $C$, lemma~\ref{lem:Galois} tells us that
\[
p^*\circ p_*(x)=\sum_{g\in G}x^g\neq 0,
\]
a contradiction. 

Thus the homotopy sheaf
\[
\pi_0(s_0\eM_s(\Z_C))=\pi_0(\eM_s(\Z_C))=\Z_C
\]
does not admit transfers, giving us the example we were seeking.

\section{Co-transfer}\label{sec:cotrans}
We recall how one uses the deformation to the normal bundle to define the ``co-transfer"
\[
(\P^1_F,1)\to (\P^1_F(\bar{x}),1)
\]
for  closed point $\bar{x}\in \A^1_F\subset \P^1_F$, with chosen generator $\bar{f}\in m_{\bar{x}}/m_{\bar{x}}^2$. For later use, we work in a somewhat more general setting: Let $R$ be a semi-local $k$-algebra, smooth and essentially of finite type over $k$, and $\bar{x}$ a regular closed subscheme of $\P^1_R\setminus\{1\}\subset \P^1_R$, such that the projection $\bar{x}\to \Spec R$ is finite. We choose a generator $\bar{f}\in m_{\bar{x}}/m_{\bar{x}}^2$, which we lift to a generator  $f$ for the ideal $m_{\bar{x}}\subset \sO_{\P^1,\bar{x}}$.  

Let $\mu:W_{\bar{x}}\to \P^1\times\A^1_R$ be the blow-up of $\P^1\times\A^1_R$ along $(\bar{x},0)$ with exceptional divisor $E$. Let $s_{\bar{x}}$, $C_0$ be the proper transforms $s_{\bar{x}}=\mu^{-1}[\bar{x}\times\A^1]$, $C_0=\mu^{-1}[\P^1\times0]$. Let $t$ be the standard parameter on $\A^1$; the rational function $f/t$ restricts to a well-defined rational parameter on $E$. We identify $E$ with $\P^1_{\bar{x}}$ by sending $s_{\bar x}\cap E$ to $0$, $C_0\cap E$ to $1$ and the section on $E$ defined by  $f/t=1$ to $\infty$. Note that this identification depends only  on $\bar{f}\in m_{\bar{x}}/m^2_{\bar{x}}$. We denote these subschemes of $E$ by $0,\infty, 1$, respectively.

We let $W_{\bar{x}}^{(s_{\bar{x}})}$, $E^{(0)}$, $(\P^1_F)^{(0)}$ be following homotopy push-outs
\begin{align*}
&W_{\bar{x}}^{(s_{\bar{x}})}:=(W_{\bar{x}},W_{\bar{x}}\setminus s_{\bar{x}}),\\
&E^{(0)}:=(E,E\setminus 0),\\
&(\P^1_R)^{(0)}:=(\P^1_R,\P^1_R\setminus 0).
\end{align*}

Since $(\A^1_{\bar{x}},0)\cong *$ in $\sH_\bullet(k)$, the respective identity maps induce isomorphisms
\begin{align*}
&(E,1)\to E^{(0)},\\
&(\P^1_R,1)\to (\P^1_R)^{(0)}.
\end{align*}
Combining with the isomorphism $(\P^1_{\bar{x}},1)\cong (E,1)$, the inclusion $E\to W_{\bar{x}}$ induces the map
\[
i_0:(\P^1_{\bar{x}},1)\to W_{\bar{x}}^{(s_{\bar{x}})}.
\]
The homotopy purity theorem of Morel-Voevodsky \cite[theorem 2.23]{MorelVoev} implies as a special case  that $i_0$ is an isomorphism in $\sH_\bullet(k)$. We prove a modification of this result. Let $s_1:=\A^1_R\times 1$. We write $W$ for $W_{\bar{x}}$, etc., when the context makes the meaning clear.

\begin{lem}\label{lem:Blowup1} The identity on $W$ induces an isomorphism 
\[
 (W,C_0\cup s_1)\to W^{(s_{\bar{x}})}.
 \]
in $\sH_\bullet(k)$.
\end{lem}

\begin{proof} As $s_1\cong\A^1_R$, with $C_0\cap s_1=0$, the inclusion $(C_0,0)\to (C_0\cup s_1,0)$ is an isomorphism in $\sH_\bullet(k)$. Thus, we need to show that $ (W,C_0)\to W^{(s_{\bar{x}})}$ is an isomorphism in $\sH_\bullet(k)$. As 
$W^{(s_{\bar{x}})}=(W,W\setminus  s_{\bar{x}})$, we need to show that $C_0\to W\setminus  s_{\bar{x}}$ is an isomorphism in $\sH(k)$.

Let $U=W\setminus s_{\bar{x}}$,  $E^U=E\cap U$, $V=U\setminus E^U$ and $C_0^V:=V\cap C_0$. We first show that the diagram
\[
\xymatrix{
C_0^V\ar[r]\ar[d]&V\ar[d]\\
C_0\ar[r]&U
}
\]
is co-cartesian in $\sH(k)$. For this, we know from Morel-Voevodsky \cite{MorelVoev} that the cofiber of $V\to U$ is isomorphic to the Thom space of the normal bundle $N$ of $E^U$ in $U$. As $E\cong \P^1_{\bar{x}}$ and $E\cap s_{\bar{x}}\to\bar{x}$ is an isomorphism, $E^U$ is isomorphic to $\A^1_{\bar{x}}$ and, as $\bar{x}$ is semi-local, $N$ is the trivial bundle. Thus we have the co-cartesian diagram
\[
\xymatrix{
V\ar[r]\ar[d]&U\ar[d]\\
{*}\ar[r]&\Sigma_{\P^1}\A^1_{\bar{x}+}.
}
\]
Similarly, we have the co-cartesian diagram
\[
\xymatrix{
C_0^V\ar[r]\ar[d]&C_0\ar[d]\\
{*}\ar[r]&\Sigma_{\P^1} 0_{\bar{x}+}.
}
\]
As the inclusion $0_{\bar{x}}\to \A^1_{\bar{x}}$ is an isomorphism in $\sH(k)$, our first diagram is co-cartesian, as desired.

Next, we note that the blow-down map $\mu:W\to \P^1\times\A^1_R$ induces isomorphisms
\begin{align*}
&V\to \P^1\times\A^1_R\setminus \bar{x}\times\A^1_R \\
&C_0\to \P^1_R\times 0\\
&C_0^V\to \P^1\times0\setminus\{(\bar{x},0)\}
\end{align*}
Thus $U$ is isomorphic in $\sH(k)$ to the homotopy push-out in the diagram
\[
\xymatrix{
 \P^1\times0\setminus\{(\bar{x},0)\}\ar[r]\ar[d]& \P^1\times\A^1_R\setminus \bar{x}\times\A^1_R \\
 \P^1_R\times 0
}
\]
But by the contradictibility of $\A^1_R$, the upper horizontal arrow is an isomorphism in $\sH(k)$, and thus $C_0\to U$ is an isomorphism in $\sH(k)$, completing the proof.
\end{proof} 

\begin{lem}\label{lem:Blowup2} The inclusion $E\to W$ induces an  isomorphism
\[
(\P^1_{\bar{x}},1)\to (W_{\bar{x}},C_0\cup s_1)
\]
in $\sH_\bullet(k)$.
\end{lem}
 
 \begin{proof} We have the commutative diagram
 \[
 \xymatrix{
 (\P^1_{\bar{x}} 1)\ar[r]\ar[rd]& (W,C_0\cup s_1)\ar[d]\\
 & W^{(s_{\bar{x}})}.
 }
 \]
 The diagonal arrow is an isomorphism in $\sH_\bullet(k)$ by Morel-Voevodsky; the vertical arrow is an isomorphism by lemma~\ref{lem:Blowup1}.
 \end{proof}
 
 One usually defines the ``co-transfer"
 \[
(\P^1_R,1)\to (\P^1_{\bar{x}},1)
 \]
 as the composition
 \[
 (\P^1_R,1)\xrightarrow{i_1}W_{\bar{x}}^{(s_{\bar{x}})}\xrightarrow{i_0^{-1}} (\P^1_{\bar{x}},1).
 \]
 Instead, we will use the composition $\co_{\bar{x},\bar{f}}$ in $\sH_\bullet(k)$,
  \[
 (\P^1_R,1)\xrightarrow{i_1}(W_{\bar{x}},C_0\cup s_1)\xrightarrow{i_0^{-1}} (\P^1_{\bar{x}},1).
 \]
 which is well-defined by lemma~\ref{lem:Blowup1}. Comparing with the usual co-transfer  via the isomorphism of lemma~\ref{lem:Blowup2},
 \[
 (W_{\bar{x}},C_0\cup s_1)\xrightarrow{\id_W} W_{\bar{x}}^{(s_{\bar{x}})},
 \]
 shows that the two co-transfer maps agree.

We examine some properties of $\co_{\bar{x},\bar{f}}$. Let $s$ be the standard parameter on $\P^1$.  

\begin{lem}\label{lem:cotrid}Take $\bar{x}=0$, $f=s$.  The map $\co_{0,\bar{s}}:(\P_F^1,1)\to (\P_F^1,1)$ is the identity in $\sH_\bullet(k)$. Similarly, 
the map $\co_{\infty,\bar{s^{-1}}}:(\P_F^1,1)\to (\P_F^1,1)$ is the identity in $\sH_\bullet(k)$
\end{lem}

\begin{proof} The assertion for $\co_{\infty,\bar{s^{-1}}}$ follows from the statement for $\co_{0,\bar{s}}$ by applying the automorphism of $(\P^1,1)$ exchanging $0$ and $\infty$. Identify $\A^1$ with  $\P^1\setminus\{1\}$, sending $0$ to $0$ and $1$ to $\infty$. This embeds the blow-up $W:=W_0$ in the blow-up $\bar{W}$ of $\P^1\times\P^1_F$ at $(0,0)$, with $i_1$  being the inclusion
\[
i_\infty:\P^1=\P^1\times\infty\to \P^1\times\P^1_F.
\]
The curve $C_0$ on $\bar{W}$ has self-intersection -1, and can thus be blown down via a morphism
\[
\rho:W\to W'.
\]

Let $\Delta\subset\P^1_F\times  \P^1\setminus\{1\}$ be the restriction of the diagonal in $\P^1_F\times\P^1$, giving us the proper transform $\mu^{-1}[\Delta]$ on $W$ and the image $\Delta'=\rho(\mu^{-1}[\Delta])$ on $W'$. Similarly, let   $s_0'=\rho(s_0)$, $s_1'=\rho(s_1)$; note that $\rho(C_0)\subset s_1'$. It is easy to check that $s_0'$, $\Delta'$ and $s_1'$ give disjoint sections of $W'\to \P^1\setminus\{1\}$, hence there is a unique isomorphism (over $\P^1\setminus\{1\}$) of $W'$ with $\P^1_F\times\P^1\setminus\{1\}$ sending $(s_0',s_1 ',\Delta')$ to $(0,1,\infty)\times\P^1\setminus\{1\}$. We have in addition the commutative diagram
\[
\xymatrix{
(\P_F^1,1)\ar[r]^-{i_0}\ar[dr]_{i_0'}&(W,C_0\cup s_1)\ar[d]^\rho&(\P^1_F,1)\ar[l]_-{i_1}\ar[dl]^{i_1'}\\
&(W',s_1')}
\]
where $i_0'$ is the canonical identification of $\P^1$ with the fiber of $W'$ over $0$, sending $(0,1,\infty)$ into $(s_0',s_1',\Delta')$, and $i_1'$ is defined similarly. Finally, the intersection $\Delta\cap E$ is the point $s/t=1$ used to define the isomorphism $E\cong\P^1_F$ in the definitiion of $\co_{0,\bar{s}}$.  It follows from lemma~\ref{lem:Blowup2} that all the morphisms in this diagram are isomorphisms in $\sH_\bullet(k)$; as $i_0^{\prime-1}\circ i_1'$ is clearly the identity, the lemma is proved.
\end{proof}

\begin{lem}\label{lem:FlatPullback} Let  $R\to R'$ be a flat extension of smooth semi-local  $k$-algebras, essentially of finite type over $k$.  Let $\bar{x}$ be a closed subscheme of $\P^1_R\setminus\{1\}$, finite and \'etale over $R$. Let $\bar{x}'=\bar{x}\times_RR'\subset \P^1_{R'}$. Let $\bar{f}$ be a generator for $m_{\bar{x}}/m^2_{\bar{x}}$, and let $\bar{f}'$ be the extension to $m_{\bar{x}'}/m^2_{\bar{x}'}$. Then the diagram 
\[
\xymatrixcolsep{40pt}
\xymatrix{
(\P^1_{R'},1)\ar[r]^{\co_{\bar{x}',\bar{f}'}}\ar[d]&(\P^1_{\bar{x}'},1)\ar[d]\\
(\P^1_{R},1)\ar[r]_{\co_{\bar{x},\bar{f}}}&(\P^1_{\bar{x}},1)
}
\]
commutes.
\end{lem}

The proof is easy and is left to the reader.

\section{Co-group structure on $\P^1$}
Let $\G_m=\A^1\setminus\{0\}$, which we consider as a pointed scheme with base-point $1$. We recall the Mayer-Vietoris square for the standard cover of $\P^1$
\[
\xymatrix{
\G_m\ar[r]^{j_\infty}\ar[d]_{j_0}&\A^1\ar[d]^{i_\infty}\\
\A^1\ar[r]_{i_0}&\P^1}
\]
Here $i_0(\A^1)=\P^1\setminus\{\infty\}$, $i_\infty(\A^1)=\P^1\setminus\{0\}$ and the inclusions are normalized by sending $1$ to $1=(1:1)$. This gives us the isomorphism in $\sH_\bullet(k)$ of $\P^1$ with the homotopy push-out in the diagram
\[
\xymatrix{
\G_m\ar[r]^{j_\infty}\ar[d]_{j_0}&\A^1\\
\A^1&}
\]
combining with the contractibility of $\A^1$ gives us the canonical isomorphism 
\[
(\P^1,1)\cong S^1\wedge\G_m.
\]

This  together with the standard co-group structure on $S^1$, 
\[
\sigma:S^1\to S^1\vee S^1
\]
makes $(\P^1,1)$ a co-group object in $\sH_\bullet(k)$; let
\[
\sigma_{\P^1}:=\sigma\wedge\id_{\G_m}:(\P^1,1)\to (\P^1,1)\vee (\P^1,1)
\]
be the co-multiplication. In this section, we discuss a more algebraic description of this structure.

The function $f:=s/(s-1)^2$ on $\P^1\setminus\{1\}$ and the deformation to the normal bundle used in the previous section gives us the collapse map
\[
\co_{\{0,\infty\},\bar{f}}:(\P^1,1)\to (\P^1,1)\vee(\P^1,1).
\]

\begin{lem}\label{lem:ComultIdent} $\co_{\{0,\infty\},\bar{f}}=\sigma_{\P^1}$ in $\sH_\bullet(k)$.
\end{lem}

\begin{proof} We first unwind the definition of $\sigma_{\P^1}$ in some detail. The isomorphism $\alpha:S^1\wedge\G_m\to (\P^1,1)$ in $\sH_\bullet(k)$ arises via a sequence of comparison maps between push-out diagrams:
\begin{align}\label{align:pushout}
\xymatrix{
\G_m\ar[r]^{j_\infty}\ar[d]_{j_0}&\A^1 \\
\A^1 &}\quad \lower20pt\hbox{$\leftarrow$}\quad
&\xymatrix{
I\times \G_m&1\times\G_m\ar[l]_{i_1}\ar[r]^{j_\infty}&1\times \A^1\\
0\times\G_m\ar[u]^{i_0}\ar[d]_{j_0}\\
0\times \A^1}\notag\\
&\hskip70pt\downarrow\\
&\xymatrix{
I\times \G_m&1\times\G_m\ar[l]_{i_1}\ar[r] &{*};\\
0\times\G_m\ar[u]^{i_0}\ar[d]\\
{*}}\notag
\end{align}
the first map is induced by the evident projections and the second by contracting $\A^1$ to $*$. Thus, the open immersion $\G_m\to\P^1$ goes over to the map $I\times\G_m\to S^1\wedge\G_m$ given by the bottom push-out diagram; as the inclusion $\{1/2\}\times\G_m\to I\times\G_m$ admits a deformation retract, we have the isomorphism
\[
\rho:(\P^1,\G_m)\to S^1\wedge\G_m\vee S^1\wedge\G_m
\]
in $\sH_\bullet(k)$, giving the commutative diagram
\[
\xymatrix{
(\P^1,1)\ar[r]^-\sim\ar[d]&S^1\wedge\G_m\ar[d]^{\sigma\wedge\id}\\
(\P^1,\G_m)\ar[r]_-\sim& S^1\wedge\G_m\vee S^1\wedge\G_m. 
}
\]
If we consider the middle push-out diagram, we find the isomorphism of $(\P^1,\G_m)$ with $(0\times\A^1,0\times\G_m)\vee(1\times\A^1,1\times\G_m)$ in $\sH_\bullet$, with the first inclusion the standard one $\A^1\setminus\{0\}\to \A^1$, and the second the inclusion $(\P^1\setminus\{0,\infty\},1)\to (\P^1\setminus\{0\},1)$. The map from the middle diagram to the last diagram furnishes the commutative diagram of isomorphisms
\[
\xymatrix{
(0\times\A^1,0\times\G_m)\vee(1\times\A^1,1\times\G_m)\ar[r]^-\beta\ar[dr]_{\vartheta}&
S^1\wedge\G_m\vee S^1\wedge\G_m\ar[d]^{\alpha\vee\alpha}\\
&(\P^1,1)\vee(\P^1,1)
}
\]
in $\sH_\bullet(k)$. Putting this all together gives us the commutative diagram in $\sH_\bullet(k)$:
\begin{equation}\label{eqn:CoMultDiag}
\xymatrix{
(\P^1,1)\ar[r]_{\rho}^\sim\ar[d]&S^1\wedge\G_m\ar[d]\\
(\P^1,\G_m)\ar[r]^\sim\ar@/_15pt/[dddr]_\delta\ar[dr]^\epsilon& (S^1\wedge\G_m,I\wedge\G_m)\ar[d]^\sim\\
&(0\times\A^1,0\times\G_m)\vee(1\times\A^1,1\times\G_m)\ar[d]_\beta^\sim\\
&S^1\wedge\G_m\vee S^1\wedge\G_m\ar[d]_{\alpha\vee\alpha}^\sim\\
& (\P^1,1)\vee(\P^1,1),
}
\end{equation}
We thus need to show that the resulting map $(\P^1,1)\to (\P^1,1)\vee(\P^1,1)$ is given by $\co_{\{0,1\},\bar{f}}$.

Let $j_0:\A^1\to\P^1$ be the standard affine neighborhood of 0, and $j_\infty :\A^1\to\P^1$ the standard affine neighborhood of $\infty$. The maps $j_0, j_\infty$ induce the isomorphisms in $\sH_\bullet(k)$
\begin{align*}
&j_0:(\A^1,\G_m)\to (\P^1,j_\infty(\A^1))\\
&j_\infty:(\A^1,\G_m)\to (\P^1,j_0(\A^1))
\end{align*}
giving together the isomorphism $\tau:(\P^1,1)\vee(\P^1,1)\to (\A^1,\G_m)\vee (\A^1,\G_m)$:
\begin{multline*}
(\P^1,1)\vee(\P^1,1) \to (\P^1,j_\infty(\A^1))\vee (\P^1,j_0(\A^1))
\xrightarrow{j_0^{-1}\vee j_\infty^{-1}}(\A^1,\G_m)\vee (\A^1,\G_m).
\end{multline*}
By comparing the maps in this composition with the push-out diagrams in \eqref{align:pushout}, we see that $\tau$ is the inverse to $\vartheta$

Let $W\to\P^1\times \A^1$ be the blow-up at $(\{0,\infty\},0)$  with exceptional divisor $E$. We have the composition of isomorphisms in $\sH_\bullet(k)$
\begin{equation}\label{eqn:eq1}
(\P^1,\G_m)\xrightarrow{i_1}
(W,W\setminus s_{\{0,\infty\}})\xleftarrow{i_0}
(E,C_0\cap E)\xleftarrow{\phi}(\P^1,1)\vee (\P^1,1).
\end{equation}
where $\phi$ is given by the isomorphism $\P^1\amalg\P^1\to E$ defined via the function $f$.

The open cover $(j_0,j_\infty):\A^1\amalg\A^1\to\P^1$ of $\P^1$ gives rise to an open cover of $W$: Let $\mu':W'\to \A^1\times\A^1$ be the blow-up at $(0,0)$, then we have the lifting of $(j_0,j_\infty)$ to the open cover
\[
(\tilde{j}_0,\tilde{j}_\infty):W'\amalg W'\to W.
\]
Letting $s'\subset W'$ be the proper transform of $0\times \A^1$ to $W'$, we have the excision isomorphism in $\sH_\bullet(k)$
\[
(\tilde{j}_0,\tilde{j}_\infty):(W',W'\setminus s')\vee (W',W'\setminus s')\to (W,W\setminus s_{\{0,\infty\}}).
\]

This extends to a commutative diagram of isomorphisms in $\sH_\bullet(k)$
\begin{equation}\label{eqn:eq2}
\xymatrixcolsep{40pt}
\xymatrix{
(\A^1,\G_m)\vee(\A^1,\G_m)\ar[d]_{i_1\vee i_1}\ar[r]^-{(j_0, j_\infty)}&(\P^1,\G_m)\ar[d]^{i_1}\\
(W',W'\setminus s')\vee (W',W'\setminus s')\ar[r]^-{(j_0, j_\infty)}&(W,W\setminus s_{\{0,\infty\}})\\
(E', E'\cap C'_0)\vee(E', E'\cap C'_0)\ar[u]^{i_0\vee i_0}\ar[r]^-{(j_0,j_\infty)}&(E,E\cap C_0)\ar[u]_{i_0}\\
(\P^1,1)\vee(\P^1,1)\ar[u]^{\phi'\vee\phi'}\ar@{=}[r]&(\P^1,1)\vee(\P^1,1).\ar[u]^\phi
}
\end{equation}
Here the map $\phi'$ is defined using the standard coordinate on $\A^1$ as generator for $m_0/m_0^2$; this is where we use the special property of the function $f$. Examining the push-out diagram \eqref{align:pushout}, we see that the map
\[
(j_0,j_\infty):(\A^1,\G_m)\vee(\A^1,\G_m)\to (\P^1,\G_m)
\]
is inverse to the map  $\epsilon$ in diagram \eqref{eqn:CoMultDiag}.

Let $W_0\to \P^1\times \A^1$ be the blow-up along $(0,0)$, $E^0$ the exceptional divisor, $C^0_0$ the proper transform of $\P^1\times0$. The inclusion $j_0$ gives us the commutative diagram
\[
\xymatrixcolsep{40pt}
\xymatrix{
&(\P^1,1)\vee(\P^1,1)\ar[d]\\
(\A^1,\G_m)\vee(\A^1,\G_m)\ar[ru]^-\vartheta\ar[d]_{i_1\vee i_1}\ar[r]^{(j_0\vee j_0)}&(\P^1,\A^1)\vee(\P^1,\A^1)\ar[d]_{i_1\vee i_1}\\
(W',W'\setminus s')\vee (W',W'\setminus s')\ar[r]^{(j_0\vee j_\infty)}&(W_0,W_0\setminus s_0)\vee (W_0,W_0\setminus s_0)\\
(E', E'\cap C'_0)\vee(E', E'\cap C'_0)\ar[u]^{i_0\vee i_0}\ar[r]^-{(j_0,j_\infty)}&(E^0,E^0\cap C^0_0)\vee (E^0,E^0\cap C^0_0)\ar[u]_{i_0\vee i_0}\\
(\P^1,1)\vee(\P^1,1)\ar[u]^{\phi'\vee\phi'}\ar@{=}[r]&(\P^1,1)\vee(\P^1,1).\ar[u]^{\phi'\vee\phi'}
}
\]
By  lemma~\ref{lem:cotrid} the composition along the right-hand side of this diagram is the identity on $(\P^1,1)\vee(\P^1,1)$, and thus the composition along the left-hand side is $\vartheta:(\A^1,\G_m)\vee(\A^1,\G_m)\to (\P^1,1)\vee(\P^1,1)$. Referring to diagram \eqref{eqn:CoMultDiag}, as $\epsilon=(j_0, j_\infty)^{-1}$, the composition along the right-hand side of \eqref{eqn:eq2} is the map $\delta$. As  the right-hand side of \eqref{eqn:eq2} is the deformation diagram used to define $\co_{\{0,\infty\},\bar{f}}$, the lemma is proved.
\end{proof}

\section{Slice localizations and co-transfer}

 In general, the co-transfer maps do not have the properties necessary to give a loop-spectrum $\Omega_{\P^1}E$ an action by correspondences. However, if we pass to a certain localization of $\SH_{S^1}(k)$ defined by the slice filtration, the co-transfer maps do respect correspondences, which will lead to the action of correspondences on $s_0\Omega_{\P^1}E$.

We have the localizing subcategory $\Sigma_{\P^1}^n\SH_{S^1}(k)$, generated (as a localizing subcategory) by objects of the form $\Sigma^n_{\P^1}E$, for $E\in \SH_{S^1}(k)$. We let $\SH_{S^1}(k)/f_n$ denote the localization of $\SH_{S^1}(k)$ with respect to  $\Sigma_{\P^1}^n\SH_{S^1}(k)$:
\[
\SH_{S^1}(k)/f_n=\SH_{S^1}(k)/\Sigma^n_{\P^1}\SH_{S^1}(k).
\]
 
\begin{rem} Pelaez has shown that there is a model structure on $\Spt_{S^1}(k)$ with homotopy category equivalent to 
$\SH_{S^1}(k)/f_{n}$; in particular, this localization of $\SH_{S^1}(k)$ does exist.
\end{rem}

\begin{lem}\label{lem:Birat} Let $V\to U$ be a dense open immersion in $\Sm/k$, $n\ge1$ an integer. Then the induced map
\[
\Sigma^n_{\P^1}V_+\to \Sigma^n_{\P^1} U_+
\]
is an isomorphism in $\SH_{S^1}(k)/f_{n+1}$.
\end{lem}

\begin{proof} We can filter $U$ by open subschemes
\[
V=U_{N+1}\subset U_N\subset\ldots\subset U_0=U
\]
such that $U_{i+1}=U_i\setminus C_i$, with $C_i\subset U_i$ smooth and having trivial normal bundle in $U_i$, of rank say $r_i$,, for $i=0,\ldots, N$. By the  Morel-Voevodsky purity theorem \cite[theorem 2.23]{MorelVoev}, the cofiber of $U_{i+1}\to U_i$ is isomorphic in $\sH_\bullet(k)$ to $\Sigma_{\P^1}^{r_i}C_{i+}$, and thus the cofiber of $\Sigma^n_{\P^1}U_{i+1+}\to\Sigma^n_{\P^1}U_{i+}$ is isomorphic to $\Sigma_{\P^1}^{r_i+n}C_{i+}$.  Since $V$ is dense in $U$, we have $r_i\ge1$ for all $i$, proving the lemma.
\end{proof}

We recall the blow-up $W:=W_{\bar{x}}\to\P^1_R\times\A^1$.

\begin{lem}\label{lem:Semiloc} Let  $m_R\subset R$ be the Jacobson radical of $R$ and let $S$ be the semi-localization of $R[t]$ with respect to the ideal $t(t-1)+m$. Let $\Spec S\to \Spec R[t]=\A^1_R$ be evident open immersion. Let $W_S$, $s_{1,S}$ be the respective fiber products of $W$, $s_1$ with $\Spec S$ over $\A^1_R$. Then the inclusion $W_S\to W$
induces an isomorphism
\[
(W_S,C_0\cup s_{1,S})\to (W,C_0\cup s_1)
\]
in $\SH_{S^1}(k)/f_2$.  
\end{lem}

\begin{proof} We note that 
\[
W\times_{\A^1_R}\A^1_R\setminus\{0\}\cong \P^1\times\A^1_R\setminus\{0\}
\]
and thus the cofiber of $W_S\to W$ is isomorphic to the cofiber of 
\[
(\P^1\times\Spec S[t^{-1}], 1)\to  (\P^1\times\A^1_R\setminus\{0\},1)
 \]
We can form $\Spec S[t^{-1}]$ from $\A^1_R\setminus\{0\}$ as a filtered projective limit of open subschemes $U_\alpha$ of $\A^1_R\setminus\{0\}$, 
by successively removing smooth closed subschemes $C_\alpha\subset U_\alpha$ from $U_\alpha$ to form $U_{\alpha+1}$; we may also assume that each $C_\alpha$ has trivial normal bundle in $U_\alpha$. The cofiber of
\[
(\P^1\times U_{\alpha+1},1)\to (\P^1\times U_\alpha,1)
\]
is thus isomorphic in $\sH_\bullet(k)$ to the pair of Thom spaces $(Th(\sO^r_{\P^1\times C_\alpha}), Th(\sO^r_{1\times C_\alpha}))$ for some $r\ge1$. As this is isomorphic to $\Sigma^r_{\P^1}(\P^1_{C_\alpha},1_{C_\alpha})\cong \Sigma^{r+1}_{\P^1}C_{\alpha+}$ this cofiber is in $\Sigma^2_{\P^1}\SH_{S^1}(k)$ for all $\alpha$. Since $\Sigma^2_{\P^1}\SH_{S^1}(k)$ is localizing, the cofiber of  $W_S\to W$ is in $\Sigma^2_{\P^1}\SH_{S^1}(k)$, as desired.
\end{proof}

\begin{lem}\label{lem:CanonThomIso} Let $W\subset U$ be a codimension $\ge r$ closed subscheme of $U\in\Sm/k$, let $w_1,\ldots, w_m$ be the generic points of $W$ of codimension $=r$ in $U$. Then in $\SH_{S^1}/f_{r+1}$ there is a canonical isomorphism
\[
(U,U\setminus W)\cong \oplus_{i=1}^m \Sigma_{\P^1}^r w_{i+}.
\]
Specifically, letting $m_i\subset\sO_{U,w_i}$ be the maximal ideal, this isomorphism is independent of any choice of isomorphism $m_i/m_i^2\cong k(w_i)^r$.
\end{lem}

\begin{proof} Let $w=\{w_1,\ldots. w_m\}$ and let $\sO_{U,w}$ denote the semi-local ring of $w$ in $U$. Let $V\subset U$ be the projective limit of open subschemes of $U$ of the form $U\setminus C$, where $C$ is a closed subset of $U$ contained in $W$. Since $V\cap W=w$, we see that $V\cap W$ has trivial normal bundle in $V$ and thus by the Morel-Voevodsky purity isomorphism \cite[{\it loc. cit}]{MorelVoev}
\[
(V,V\setminus V\cap W)\cong \Sigma_{\P^1}^r w_+\cong \oplus_{r=1}^m\Sigma_{\P^1}w_{i+}.
\]
On the other hand, we can write $V$ as a filtered projective limit of open subschemes $U_\alpha$ of $U$, with $U_{\alpha+1}=U_\alpha\setminus C_\alpha$, for some closed subset $C_\alpha$ of $W$ which is smooth, contains no $w_i$, and has trivial normal bundle in $U_\alpha$. In particular, $C_\alpha$ has codimenison $\ge r+1$ in $U_\alpha$, and thus
\[
(U_\alpha,U_{\alpha+1})\cong Th(\sO^n_{C_\alpha})\cong \Sigma^n_{\P^1}C_{\alpha+}
\]
is in $\Sigma_{\P^1}^{r+1}\SH_{S^1}(k)$. Therefore the map
\[
(V,V\setminus V\cap W)\to (U,U\setminus W)
\]
is an isomorphism in $\SH_{S^1}(k)/f_{r+1}$. 

We need only verify that the resulting isomorphism $(U,U\setminus W)\cong \oplus_{r=1}^m\Sigma_{\P^1}w_{i+}$ is independent of any choices. Letting $\sO$ denote the Henselization of $w$ in $V$, we have the canonical excision isomorphism
\[
(V,V\setminus V\cap W)\cong (\Spec\sO,\Spec\sO\setminus w)
\]
A choice of isomorphism
\[
m_w/m_w^2\cong k(w)^r
\]
then gives the isomorphism in $\sH_\bullet(k)$
\[
 \Sigma_{\P^1}^r w_+\cong (\Spec\sO,\Spec\sO\setminus w);
\]
this choice of isomorphism is thus the only choice involved in constructing our isomorphism $(U,U\setminus W)\cong \oplus_{i=1}^m \Sigma_{\P^1}^r w_{i+}$. Explicitly, the choice of isomorphism $m_w/m_w^2\cong k(w)^r$ is reflected in the isomorphism $(\Spec\sO,\Spec\sO\setminus w)\cong \Sigma_{\P^1}^r w_+$ through the identification of the exceptional divisor of the blow-up of $V\times\A^1$ along $w\times0$ with $\P^r_w$.

Let $\Iso$ denote the $k(w)$ scheme of isomorphisms $m_w/m_w^2\cong k(w)^r$. We thus have a canonical  morphism in $\sH_\bullet(k)$
\begin{equation}\label{eqn:Canon}
\Sigma^r_{\P^1}\Iso_+\to (\Spec\sO,\Spec\sO\setminus w)  
\end{equation}
such that, for each $k(w)$-point $\alpha$ of $\Iso$, the composition
\[
\Sigma^r_{\P^1}w_+\xrightarrow{\Sigma^r_{\P^1}\alpha}\Sigma^r_{\P^1}\Iso_+\to (\Spec\sO,\Spec\sO\setminus w)  
\]
is the isomorphism in $\sH_\bullet(k)$ described above.

We note that $\Iso$ is a trivial principal homogeneous space for $\GL_{r/k(w)}$, hence isomorphic to an open subscheme   of $\A^{r^2}_{k(w)}$, and thus $(\A^{r^2},\Iso)$ is in $\Sigma^1_{\P^1}\SH_{S^1}(k)$. We thus have the  isomorphism in $\SH_{S^1}(k)/f_{r+1}$
\[
\Sigma^r_{\P^1}\Iso_+\cong \Sigma^r_{\P^1}\A^{r^2}_{w+}\cong \Sigma^r_{\P^1}w_{+},
\]
from which it easily follows that the map $\Sigma^r_{\P^1}\alpha$ in $\SH_{S^1}(k)/f_{r+1}$ is independent of the choice of $\alpha$, completing the proof.
\end{proof}

\begin{rem}\label{rem:Canon} As a particular case, we have a canonical isomorphism
\[
\psi_{0,\infty}:(\P^1,\G_m)\to (\P^1,1)\vee(\P^1,1)
\]
in $\SH_{S^1}(k)/f_2$, independent of any choice of generator for $m_{0,\infty}/m_{0,\infty}^2$. This simplifies our description of the co-multiplication on $(\P^1,1)$ at least if we work in $\SH_{S^1}(k)/f_2$, as being given by the composition
\[
(\P^1,1)\to (\P^1,\G_m)\xrightarrow{\psi_{0,\infty}}(\P^1,1)\vee(\P^1,1).
\]

Similarly, for each $a\in\P^1(F)$, we have a canonical isomorphism $\psi_a:(\P^1_F,\P^1_F\setminus\{a\})\to (\P^1_F,1)$ in $\SH_{S^1}(k)/f_2$; by lemma~\ref{lem:cotrid}, the composition 
\[
(\P^1_F,1)\to (\P^1_F,\P^1_F\setminus\{a\})\xrightarrow{\psi_0} (\P_F^1,1)
\]
in $\SH_{S^1}(k)/f_2$ is the identity.
\end{rem}

\begin{lem}\label{lem:Mult} Let $\mu_n:(\P^1_F,\infty)\to (\P^1_F,\infty)$ be the $\mu_n(1:t)=(1,t^n)$. Assume the characteristic of $k$ is prime to $n!$. Then in $\SH_{S^1}(k)/f_2$, $\mu_n$ is multiplication by $n$.
\end{lem}

\begin{proof} The proof goes by induction on $n$, starting with $n=1,2$. The case $n=1$ is trivial. For $n=2$, we have the commutative diagram
\[
\xymatrix{
(\P^1_F,\infty)\ar[r]\ar[d]_{\mu_2}&(\P^1_F,\P^1_F\setminus\{\pm1\})\ar[d]^{\mu_2}\\
(\P^1_F,\infty)\ar[r]&(\P^1_F,\P^1_F\setminus\{1\})
}
\]
The bottom horizontal arrow is an isomorphism in $\sH_\bullet(k)$; using remark~\ref{rem:Canon} we see that this diagram gives us the factorization of $\mu_2$ (in $\SH_{S^1}(k)/f_2$) as
\[
(\P^1_F,\infty)\xrightarrow{\sigma}(\P^1_F,\infty)\vee(\P^1_F,\infty)\xrightarrow{\id\vee\id}(\P^1_F,\infty).
\]
Here $\sigma$ is the co-multiplication (using $\infty$ instead of $1$ as base-point).  Since $[(\id\vee\id)\circ\sigma]^*$ is multiplication by 2, this takes care of the case $n=2$.

In general, we consider the map $\rho_n:(\P^1,\infty)\to (\P^1,\infty)$ sending $(1:t)$ to $(1:w):=(1:t^n-t^{n-1}+1)$. As above, we localize around $w=1$. Note that $\rho_n^{-1}(1)=\{0,1\}$; we replace the target $\P^1$ with the Henselization $\sO$ at $w=1$, and see that $\P^1\times_{\rho_n}\sO$ breaks up into two components. On the component containing $0$, the map $\rho_n$ is isomorphic to a Hensel local version of $\rho_{n-1}$, and on the component containing $1$, the map $\rho_n$ is isomorphic to the identity. More explicitly, we have the following commutative diagram (in $\SH_{S^1}(k)/f_2$)
\[
\xymatrix{
(\P^1_F,\infty)\ar[r]\ar[d]_{\mu_n}&(\P^1_F,\P^1_F\setminus\{0,1\})\ar[d]_{\mu_n}\ar[r]^-\sim&(\P_F^1,\infty)\vee(\P_F^1,\infty)\ar[d]^{\mu_{n-1}\vee\id}\\
(\P^1_F,\infty)\ar[r]&(\P^1_F,\P^1_F\setminus\{1\})\ar[r]_-\sim&(\P^1_F,\infty)
}
\]
As the upper row is the co-multiplication (in $\SH_{S^1}(k)/f_2$), our induction hypothesis shows that $\rho_n$ (in 
$\SH_{S^1}(k)/f_2$)  is multiplication by $n$. On the other hand, we may form the family of morphisms
\[
\rho_n(s):(\P^1\times\A^1,\infty\times\A^1)\to (\P^1\times\A^1,\infty\times\A^1)
\]
sending $(t_0:t_1,s)$ to $(t_0^n:t_1^n-st_0t_1^{n-1}+t_0^n)$. By homotopy invariance, we have $\rho_n(0)=\rho_n(1)$, and the induction goes through.
\end{proof}

While we are on the subject, we might as well show

\begin{prop} The co-group $((\P^1,1),\sigma_{\P^1})$ in $\SH_{S^1}(k)/f_2$ is co-commutative.
\end{prop}

\begin{proof} In $\P^1\times\P^1$, consider the diagonal $\Delta$ and anti-diagonal $\Delta'$, defined by $x_0y_1-x_1y_0=0$ and $x_1y_1-x_0y_0=0$, respectively. In affine coordinates, these are $y=x$ and $y=1/x$, respectively, hence 
$\Delta\cap \Delta'$ consists of the two points $(1,1)$ and $(-1,-1)$. Thus, if we restrict to $U\times\P^1$, $U=\P^1\setminus\{\pm1\}$, the subscheme $\Delta_U\cup\Delta'_U$ of $U\times \P^1$ is \'etale and finite over $U$ and is disjoint from $U\times1$. This gives us the map
\[
\tilde{\sigma}: (U\times\P^1,U\times1)\to (U\times\P^1, U\times\P^1\setminus \Delta_U\cup\Delta'_U)
\]
The composition $U\times\P^1\to \A^1\times\P^1\to \P^1$ shows that the projection induces an isomorphism
\[
(U\times\P^1,U\times1)\to (\P^1,1)
\]
in $\SH_{S^1}(k)/f_2$, while on the other side, we have the Thom isomorphism
\[
 (U\times\P^1, U\times\P^1\setminus \Delta_U\cup\Delta'_U)\cong Th(N_{\Delta_U})\vee Th(N_{\Delta'_U})\cong \Sigma_{\P^1}U_+\vee \Sigma_{\P^1}U_+,
 \]
 which in turn is isomorphic to $(\P^1,1)\vee(\P^1,1)$ in $\SH_{S^1}(k)/f_2$. As the restriction of $\tilde{\sigma}$ to $0\in U$ is the map $\sigma$ and the restriction to $\infty\in U$ is $\sigma$ followed by the exchange isomorphism 
 \[
 \tau: (\P^1,1)\vee  (\P^1,1)\to  (\P^1,1)\vee  (\P^1,1),
 \]
 (using remark~\ref{rem:Canon} to avoid specifying the choice of trivialization in the Thom isomorphism) we have proven the co-commutativity.
 \end{proof}

We now return to our study of properties of the co-transfer map in $\SH_{S^1}(k)/f_2$. We already know that, for a given closed point $\bar{x}\in\P^1_F$, the map
\[
\co_{\bar{x},\bar{f}}:(\P^1_F,1)\to (\P^1_{F(\bar{x})},1)
\]
in $\SH_{S^1}(k)/f_2$ is independent of the choice of generator $f\in m_{\bar{x}}/m^2_{\bar{x}}$; we denote this map by $\co_{\bar{x}}$.

Suppose we have a semi-local smooth $k$-algebra $A$, essentially of finite type, and a finite extension $A\to B$, with $B$ smooth over $k$. Suppose further that $B$ is generated as an $A$-algebra by a single element $x\in B$:
\[
B=A[x].
\]
We say in this case that $B$ is a {\em simply generated} $A$-algebra.

Let $f\in A[T]$ be the minimal monic polynomial of $x$, giving us the point $\bar{x}'$ of $\A_A^1=\Spec A[T]$ with ideal $(f)$. We identify $\A^1_A$ with $\P^1_A\setminus\{1\}$ as usual, giving us the subscheme $\bar{x}$ of $\P^1_A\setminus\{1\}$, smooth over $k$ and  finite   over $\Spec A$, in fact, canonically isomorphic to $\Spec B$ over $\Spec A$ via the choice of generator $x$.   Let 
\[
\phi_x:\bar{x}\to \Spec B
\]
be this isomorphism. We let  $\bar{f}$ be the generator of $m_{\bar{x}}/m^2_{\bar{x}}$ determined by $f$. Via the composition
\[
(\P^1_A,1)\xrightarrow{\co_{\bar{x},\bar{f}}}(\P^1_{\bar{x}},1)\xrightarrow{\phi_x\times\id}(\P^1_B,1)
\]
we have the morphism 
\[
\co_x:(\P^1_A,1)\to (\P^1_B,1)
\]
in $\sH_\bullet(k)$.

\begin{lem}\label{lem:Independence}Suppose that $\Spec B\to \Spec A$ is \'etale over each generic point of $\Spec A$. Then the map $\co_x:(\P^1_A,1)\to (\P^1_B,1)$ in $\SH_{S^1}(k)/f_2$ is independent of the choice of generator $x$ for $B$ over $A$. We write $\co_{B/A}$ for  $\co_x$.
\end{lem}

\begin{proof}   We use a deformation argument; we first localize to reduce to the case of an \'etale extension $A\to B$. For this, let $a\in A$ be a non-zero divisor, and let $x$ be a generator for $B$ as an $A$-algebra. Then $x$ is a generator for $B[a^{-1}]$ as an $A[a^{-1}]$-algebra and we have the commutative diagram
\[
\xymatrix{
\P^1_{A[a^{-1}]}\ar[r]\ar[d]_{\co_x}&\P^1_A\ar[d]^{\co_x}\\
\P^1_{B[a^{-1}]}\ar[r]&\P^1_B,
}
\]
with horizontal arrows isomorphisms in $\SH_{S^1}(k)/f_2$. Thus, we may assume that $A\to B$ is \'etale.

Suppose we have generators $x\neq x'$ for $B$ over $A$; let $d=[B:A]$.  Let $s$ be an indeterminate, let $x(s)=sx+(1-s)x'\in B[s]$, and consider the extension $\tilde{B}_s:=A[s][x(s)]$ of $A[s]$, considered as a subalgebra of $B[s]$.  Clearly $\tilde{B}_s$ is finite over $A[s]$.

Let $m_A\subset A$ be the Jacobson radical, and let $A(s)$ be the localization of $A[s]$ at the ideal $(m_AA[s]+s(s-1))$. In other words, $A(s)$ is the semi-local ring of the set of closed points $\{(0,a), (1,a)\}$ in $\A^1\times\Spec A$, as $a$ runs over the closed points of $\Spec A$. Define $B(s):=B\otimes_AA(s)$ and $B_s:=\tilde{B}_s\otimes_AA(s)\subset B(s)$. Let $y=(1,a)$ be a closed point of $A(s)$, with maximal ideal $m_y$, and let $x_y$ be the image of $x$ in $B(s)/m_yB(s)$. Clearly $x_y$ is in the image of $B_s\to B(s)/m_yB(s)$, hence $B_s\to B(s)/m_yB(s)$ is surjective. Similarly, 
$B_s\to B(s)/m_yB(s)$ is surjective for all $y$ of the form $(0,a)$; by Nakayama's lemma $B_s=B(s)$. Also, $B(s)$ and $A(s)$ are regular and $B(s)$ is finite over $A(s)$, hence $B(s)$ is flat over $A(s)$ and thus $B(s)$ is a free $A(s)$-module of rank $d$. Finally, $B(s)$ is clearly unramified over $A(s)$, hence $A(s)\to B(s)$ is \'etale.

Using Nakayama's lemma again, we see that $B(s)$ is generated as an $A(s)$ module by $1, x(s), x(s)^2,\ldots, x(s)^{d-1}$. It follows that  $x(s)$ satisfies a monic   polynomial  equation of degree d over $A(s)$, thus $x(s)$ admits a monic minimal polynomial $f_s$ of degree $d$ over $A(s)$. Sending $T$ to $x(s)$ defines an isomorphism
\[
\phi_s:A(s)[T]/(f_s)\to B(s).
\]
We let $\bar{x}_s\subset\A^1_{A(s)}=\P^1_{A(s)}\setminus\{1\}$ be the closed subscheme of $\P^1_{A(s)}$ corresponding to $f_s$; the isomorphism $\phi_s$ gives us the isomorphism 
\[
\phi_s:\bar{x}_s\to \Spec B(s).
\]
Thus, we may define the map
\[
\co_{x(s)}:(\P^1_{A(s)},1)\to (\P^1_{B(s)},1)
\]
giving us the commutative diagram
\[
\xymatrix{
(\P^1_{A},1)\ar[d]_{\co_{x'}}\ar[r]^-{i_0}&(\P^1_{A(s)},1)\ar[d]_{\co_{x(s)}}&(\P^1_A,1)\ar[d]^{\co_x}\ar[l]_-{i_1}\\
(\P^1_B,1)\ar[r]_-{i_0}&(\P^1_{B(s)},1)&(\P^1_B,1)\ar[l]^-{i_1}
}
\]
By lemma~\ref{lem:Birat} and a limit argument,  the map $(\P^1_{A(s)},1)\to (\P^1_{A[s]},1)$ is an isomorphism in $\SH_{S^1}(k)/f_2$. By homotopy invariance, it follows that the maps $i_0, i_1$ are isomorphisms in $\SH_{S^1}(k)/f_2$, inverse to the map 
$(\P^1_{A(s)},1)\to (\P^1_{A},1)$ induced by the projection $\Spec A(s)\to \Spec A$. Therefore $\co_{x'}=\co_x$, as desired.
\end{proof}

\begin{lem}\label{lem:CoTrId} $\co_{A/A}=\id_{(\P^1_A,1)}$.
\end{lem}

\begin{proof} We may choose $1$ as the generator for $A$ over $A$, which gives us the point $\bar{x}=0\in\P^1_A$. The result now follows from lemma~\ref{lem:cotrid}.
\end{proof}

\begin{lem}\label{lem:CoTrComp}  Let $A\to C$ be a finite simply generated extension and $A\subset B\subset C$ a sub-extension, with $B$ also simply generated over $A$. We suppose that $A$, $B$ and $C$ are smooth over $k$, and that $A\to B$ and $A\to C$  are \'etale over each generic point of $\Spec A$, and $B\to C$ is \'etale over each generic point of $\Spec B$. Then
\[
\co_{C/A}=\co_{C/B}\co_{B/A}.
\]
\end{lem}

\begin{proof} This is another deformation argument. As in the proof of lemma~\ref{lem:Independence}, we may assume that $A\to B$, $B\to C$ and $A\to C$ are \'etale extensions. Let $y$ be a generator for $C$ over $A$, $x$ a generator for $B$ over $A$. These generators give us corresponding closed subschemes $\bar{y},\bar{x}\subset \P^1_A$ and  $\bar{y}_B\subset \P^1_B$. Let 
$y(s)=sy+(1-s)x$, giving $\bar{y}(s)\subset \P^1_{A(s)}$.  Note that $\bar{y}(1)=\bar{y}$, $\bar{y}(0)_\red=\bar{x}$

As in the proof of lemma~\ref{lem:Independence}, the element $y(s)$ of $C(s)$ is a generator over $A(s)$ after localizing at the points of $\Spec A(s)$ lying over $s=1$. The subscheme $\bar{y}(s)$ in a neighborhood of $s=0$ is not in general regular, hence $y(s)$ is not a generator of $C(s)$ over $A(s)$. However, let $\mu:W:=W_{\bar{x}}\to\P^1\times\A^1$ be the blow-up along $\{(\bar{x},0)\}$, and let $\tilde{y}\subset W_{A(s)}$ be the proper transform $\mu^{-1}[\bar{y}]$. An elementary local computation shows that this blow-up resolves the singularities of $\bar{y}(s)$, and that $\tilde{y}$ is \'etale over $A(s)$; the argument used in the proof of lemma~\ref{lem:Independence} goes through to show that $A(s)(\tilde{y})\cong C(s)$. In addition, let $C_0$ be the proper transform to $W_{A(s)}$ of $\P^1\times0$ and $E$ the exceptional divisor, then $\tilde{y}(0)$ is disjoint from $C_0$. Finally, after identifying $E$ with $\P^1_{A[\bar{x}]}$ (using a the monic minimal polynomial of $x$ as a generator for $m_{\bar{x}}$), we may consider $\tilde{y}(0)$ as a closed subscheme of $\P^1_B$; the isomorphism $A(s)(\tilde{y})\cong C(s)$ leads us to conclude that $A(\tilde{y}(0))=B(\tilde{y})= C$. By lemma~\ref{lem:Independence}, we may use $\tilde{y}$ to define $\co_{C/B}$.

The map $\co_{C/A}$ in $\SH_{S^1}(k)/f_2$ is defined via the diagram
\[
(\P^1_A,1)\to (\P^1_A,\P^1_A\setminus\bar{y})\cong (\P^1_C,1)
\]
where the various choices involved lead to equal maps. The inclusions $i_1:\P^1_A\to W_{A(s)}$,
$i_0:\P^1_{A[\bar{x}]} \to W_{A(s)}$ induce  isomorphisms (in $\SH_{S^1}(k)/f_2$)
\[
(\P^1_A,\P^1_A\setminus\bar{x})\cong (W_{A(s)}, W_{A(s)}\setminus\tilde{y}(s))
\cong (\P^1_{A[\bar{x}]}, \P^1_{A[\bar{x}]}\setminus \tilde{y}(0)).
\]
As in the proof of lemma~\ref{lem:Independence}, we can use  homotopy invariance to see that $\co_{C/A}$ is also equal to the composition
\begin{multline*}
(\P^1_A,1)\to (\P^1_A,\P^1_A\setminus\bar{y})\xrightarrow{i_1}(W_{A(s)}, W_{A(s)}\setminus\tilde{y}(s))\\
\xrightarrow{i_0^{-1}}
 (\P^1_{A[\bar{x}]}, \P^1_{A[\bar{x}]}\setminus \tilde{y}(0))
\cong (\P^1_{C},1).
\end{multline*}

Now   let $s_{1A(s)}$ be the transform to $W_{A(s)}$ of the 1-section. The above factorization of $\co_{C/A}$ shows that this map is also equal to the composition
\[
(\P^1_A,1)\xrightarrow{i_1}(W_{A(s)}, C_{0}\cup s_{1 A(s)})
\xrightarrow{i_0^{-1}}
(\P^1_{A[\bar{x}]},1)\to
(\P^1_{A[\bar{x}]}\setminus \tilde{y}(0))
\cong (\P^1_{C},1).
\]
Using lemma~\ref{lem:Semiloc}, this latter composition is $\co_{C/B}\circ (\co_{B/A})$, as desired.
\end{proof}

\begin{rem}\label{rem:Additive} 1. Suppose we have simply generated finite generically \'etale extensions $A_1\to B_1$, $A_2\to B_2$, with $A_i$ smooth, semi-local and essentially of finite type over $k$. Then
\[
\co_{B_1\times B_2/A_1\times A_2}=\co_{B_1/A_1}\vee\co_{B_2/A_2}
\]
where we make the   evident identification $(\P^1_{B_1\times B_2},1)=(\P^1_{B_1},1)\vee(\P^1_{B_2},1)$ and similarly for $A_1, A_2$.\\
\\
2. Let $B_1, B_2$ be simply generated finite generically \'etale $A$ algebras and let $B=B_1\times B_2$. As a special case of lemma~\ref{lem:CoTrComp}, we have
\[
\co_{B/A}=(\co_{B_1/A}\vee\co_{B_2/A})\circ\sigma_{\P^1_A}
\]
Indeed, we may factor the extension $A\to B$ as $A\xrightarrow{\delta} A\times A\to B_1\times B_2=B$. We then use (1) and note that $\sigma_{\P^1_A}=\co_{A\times A/A}$ by lemma~\ref{lem:ComultIdent}.
\end{rem}

Next, we make a local calculation. Let $(A,m)$ be a local ring of essentially of finite type and smooth over $k$. Let $s\in m$ be a parameter and let $B=A[T]/T^n-s$ and let $t\in B$ be the image of $T$. Set $Y=\Spec B$, $X=\Spec A$, $Z=\Spec A/(s)$, $W=\Spec B/(t)$; the extension $A\to B$ induces an isomorphism $\alpha:W\xrightarrow{\sim}Z$. We write $\co_{Y/X}$ for $\co_{B/A}$, etc. This gives us the diagram in $\SH_{S^1}(k)/f_2$
\[
\xymatrix{
\P^1_Z\ar[r]^{i_Z}&\P^1_X\ar[d]^{\co_{Y/X}}\\
\P^1_W\ar[r]_{i_W}\ar[u]^\alpha&\P^1_Y.
}
\]

\begin{lem}\label{lem:TotRam} Suppose that $n$ is prime to $\Char k$. In $\SH_{S^1}(k)/f_2$ we have
\[
 \co_{Y/X}\circ i_Z\circ \alpha=n\times i_W.
\]
\end{lem}

\begin{proof} First, suppose we have a Nisnevich neighborhood $f:X'\to X$ of $Z$ in $X$, giving us the Nisnevich neighborhood $g:Y':=Y\times_XX'\to Y$ of $W$ in $Y$. As
\[
\co_{Y/X}\circ f=g\circ \co_{Y'/X'}
\]
we may replace $X$ with $X'$, $Y$ with $Y'$. Similarly, we reduce to the case of $A$ a Hensel DVR, i.e., the Henselization of $0\in\A^1_F$ for some field $F$, $Z=W=0$, with $s$ the image in $A$ of the canonical coordinate on $\A^1_F$.

The map $\co_{Y/X}$ is defined by the closed immersion 
\[
Y\xrightarrow{i_Y} \A^1_X=\P^1_X\setminus\{1\}\subset \P^1_X
\]
where $i_Y$ is the closed subscheme of $\A^1=\Spec A[T]$ defined by $T^n-s$, together with the isomorphism
\[
(\P^1_X,\P^1_X\setminus Y)\cong \P^1_Y
\]
furnished by the blow-up $\mu:W_Y\to \A^1\times\A^1_X$ of $\A^1\times\A^1_X$ along $(Y,0)$. The composition $\co_{Y/X}\circ i_Z\circ \alpha$ is given by the composition
\begin{multline*}
(\P^1_W,1)\cong (\P^1_W,\P^1_W\setminus \{0\})\xrightarrow{\alpha}(\P^1_Z,\P^1_Z\setminus \{0\})\\
\xrightarrow{i_Z}(\P^1_X,\P^1_X\setminus \{0\})\xleftarrow{\id}(\P^1_X,1)\to (\P^1_X,\P^1_X\setminus Y)\cong (\P^1_Y,1).
\end{multline*}
In both cases, the isomorphisms (in $\SH_{S^1}(k)/f_2$) are independent of a choice of the respective defining equation. Let   $U\to \P^1_X$ be the Hensel local neighborhood of $(0,0)$ in $\P^1_X$,  $\Spec\sO^h_{\P^1_X,(0,0)}$, and let $U_Z\subset U$ be the fiber of $U$ over $Z$, i.e., the subscheme $s=0$.  We may use excision to rewrite the above description of $\co_{Y/X}\circ i_Z\circ \alpha$ as a composition  as
\[
(\P^1_W,1)\cong   (U_Z,U_Z\setminus\{(0,0)\})\xrightarrow{i_Z}(U,U\setminus Y)\cong (\P^1_Y,1).
\]
Similarly, letting $i_0:X\to X\times\P^1$ be the 0-section,  the map $i_W$ may be given by the composition
\[
(\P^1_W,1) \cong (X,X\setminus Z)\xrightarrow{i_0} (U, U\setminus Y)\cong (\P^1_Y,1);
\]
again, the isomorphisms in  $\SH_{S^1}(k)/f_2$ are independent of choice of defining equations.

We change coordinates in $U$ by the isomorphism $(s,t)\mapsto (s-t^n, t)$. This transforms $Y$ to the subscheme $s=0$, is the identity on the 0-section, and transforms $s=0$ to the graph of $t^n+s=0$. Replacing $s$ with $-s$, we have just switched the roles of $Y$ and $U_Z$. Let 
\[
\phi:U_Z\to U
\]
be the map  $\phi(t)=(t^n,t)$. After making our change of coordinates, the map  $\co_{Y/X}\circ i_Z\circ \alpha$ is identified with
\[
(\P^1_W,1)\cong   (U_Z,U_Z\setminus\{(0,0)\})\xrightarrow{\phi}(U,U\setminus U_Z)\cong (\P^1_Y,1)
\]
while the description of $i_W$ becomes 
\[
(\P^1_W,1) \cong (X,X\setminus Z)\xrightarrow{i_0} (U, U\setminus U_Z)\cong (\P^1_Y,1);
\]

We now construct an $\A^1$-family of maps $(U_Z,U_Z\setminus\{(0,0)\})\to (U,U\setminus U_Z)$. Let
\[
\Phi:U_Z\times\A^1\to U
\]
be the map   $\Phi(t, v)=(t^n,vt)$. Note that $\Phi$ defines a map of pairs
\[
\Phi:(U_Z, U_Z\setminus\{0\})\times\A^1\to (U,U\setminus U_Z).
\]
Clearly $\Phi(-,1)=\phi$ while $\Phi(-,0)$ factors as
\[
U_Z\xrightarrow{\mu_n}U_Z\xrightarrow{\beta}X\xrightarrow{i_0}U
\]
where $\mu_n$ is the map  $t\mapsto t^n$ and $\beta$ is the isomorphism $\beta(t)=s$. Thus, we can rewrite $\co_{Y/X}\circ i_Z\circ \alpha$ as
\[
(\P^1_W,1)\cong   (X, X\setminus Z)\xrightarrow{\mu_n} (X, X\setminus Z)\xrightarrow{i_0}(U,U\setminus U_Z)\cong (\P^1_Y,1)
\]

We identify $X$ with the Hensel neighborhood of $0$ in $\P^1_Z$. Using excision again, we have the commutative diagram in $\sH_\bullet(k)$
\[
\xymatrix{
(X, X\setminus Z)\ar[r]^{\mu_n}\ar[d]&(X, X\setminus Z)\ar[d]\\
(\P^1_Z,\P^1_Z\setminus \{0\})\ar[r]^{\mu_n}&(\P^1_Z,\P^1_Z\setminus\{0\})\\
(\P^1_Z,\infty)\ar[r]^{\mu_n^Z}\ar[u]&(\P^1_Z,\infty)\ar[u]
}
\]
where the vertical arrows are all isomorphisms. By lemma~\ref{lem:Mult} the bottom map is multiplication by $n$, which completes the proof.
\end{proof}

\begin{lem}\label{lem:Unram} Let $A\to B$ be a  finite simple \'etale extension. Let $X=\Spec A$, $Y=\Spec B$, let  $i_x:x\to X$ be the closed point of $X$ and $i_y:y\to Y$ the inclusion of $y:=x\times_XY$. Then
\[
\co_{Y/X}\circ i_x=i_y\circ \co_{y/x}.
\]
\end{lem}

\begin{proof} Take an embedding of $Y$ in $\A^1_X=\P^1_X\setminus\{1\}\subset\P^1_X$; the fiber of $Y\to \A^1_X$ over $x\to X$ is thus an embedding  $y\to \A^1_x=\P^1_x\setminus\{1\}\subset\P^1_x$. The result follows easily from the commutativity of the diagram
\[\xymatrix{
\P^1_x\setminus y \ar[r]\ar[d]&\P^1_x\ar[d]\\
\P^1_X\setminus Y\ar[r]&\P^1_X
}
\]
\end{proof}

\begin{prop} \label{prop:LocalCyc} Let $A\to B$ be a finite generically \'etale extension, with $A$ a DVR and $B$ a semi-local principal ideal ring. Let $X=\Spec A$, $Y=\Spec B$, let  $i_x:x\to X$ be the closed point of $X$ and $i_y:y\to Y$ the inclusion of $y:=x\times_XY$. Write $y=\{y_1,\ldots, y_r\}$, with each $y_i$ irreducible. Let $n_i$ denote the ramification index of $y_i$; suppose that each $n_i$ is prime to $\Char k$.  Then
\[
\co_{Y/X}\circ i_x=\sum_{i=1}^r n_i\cdot i_{y_i}\circ \co_{y_i/x}.
\]
\end{prop}

\begin{proof} By passing to the Henselization $A\to A^h$, we may assume $A$ is Hensel. By remark~\ref{rem:Additive}(2), we may assume that $r=1$. Let $A\to B_0\subset B$ be the maximal unramified subextension. As $\co_{B/A}=\co_{B/B_0}\circ\co_{B_0/B}$, we reduce to the two cases $A=B_0$, $B=B_0$. We note that a finite separable extension of Hensel DVRs $A\to B$ with trivial residue field extension degree is isomorphic to an extension of the form $t^n=s$ for some $s\in m_A\setminus m_A^2$. Thus,  the  first case is lemma~\ref{lem:TotRam}, the second is lemma~\ref{lem:Unram}.
\end{proof}

Consider the functor
\[
(\P^1_{?},1):\Sm/k\to \SH_{S^1}(k)/f_2
\]
sending $X$ to $(\P^1_X,1)\in\SH_{S^1}(k)/f_2$, which we consider as a $\SH_{S^1}(k)/f_2$-valued presheaf on $\Sm/k^\op$ (we could also write this functor as $X\mapsto \Sigma_{\P^1}X_+$). We proceed to extend $(\P^1_{?},1)$ to a presheaf on $\SmCor(k)^\op$; we will assume that $\Char k=0$, so we do not need to worry about inseparability.

We first define the action on the generators of $\Hom_{\SmCor}(X,Y)$, i.e., on irreducible $W\subset X\times Y$ such that $W\to X$ is finite and surjective over some component of $X$. As $\SH_{S^1}(k)/f_2$ is an additive category, it suffices to consider the case of irreducible $X$.   Let $U\subset X$ be a dense open subscheme. Then the map $(\P^1_U,1)\to (\P^1_X,1)$ induced by the inclusion is an isomorphism in $\SH_{S^1}(k)/f_2$. We may therefore define the morphism
\[
(\P^1_?,1)(W):(\P^1_X,1)\to (\P^1_Y,1)
\]
in $\SH_{S^1}(k)/f_2$ as the composition
\[
(\P^1_X,1)\cong (\P^1_{k(X)},1)\xrightarrow{\co_{k(W)/k(X)}} (\P^1_{k(W)},1)\xrightarrow{p_2}(\P^1_Y,1).
\]
We extend to linearity to define $(\P^1_?,1)$ on $\Hom_{\SmCor}(X,Y)$. 

Suppose that $\Gamma_f\subset X\times Y$ is the graph of a morphism $f:X\to Y$. It follows from lemma~\ref{lem:CoTrId} that 
$\co_{k(\Gamma_f)/k(X)}$ is the inverse to the isomorphism $p_1:(\P^1_{k(\Gamma_f)},1)\to (\P^1_{k(X)},1)$. Thus, the composition
\[
(\P^1_{k(X)},1)\xrightarrow{\co_{k(\Gamma_f)/k(X)}} (\P^1_{k(\Gamma_f)},1)\xrightarrow{p_2}(\P^1_Y,1)
\]
is the map induced by the restriction of $f$ to $\Spec k(X)$. Since $(\P^1_{k(X)},1)\to (\P^1_X,1)$ is an isomorphism in  $\SH_{S^1}(k)/f_2$, it follows that $(\P^1_?,1)(\Gamma_f)=f$, i.e., our definition of  $(\P^1_?,1)$ on $\Hom_{\SmCor}(X,Y)$ really is an extension of its definition on 
$\Hom_{\Sm/k}(X,Y)$.

The main point is to check functoriality.

\begin{lem} For $\alpha\in \Hom_{\SmCor}(X,Y)$, $\beta\in \Hom_{\SmCor}(Y,Z)$, we have
\[
(\P^1_?,1)(\beta\circ\alpha)=(\P^1_?,1)(\beta)\circ(\P^1_?,1)(\alpha)
\]
\end{lem}

\begin{proof} It suffices to consider the case of irreducible finite correspondences $W\subset X\times Y$, $W'\subset Y\times Z$.  If $W$ is the graph of a flat morphism, the result follows from lemma~\ref{lem:FlatPullback}.

As the action of correspondences is defined at the generic point, we may replace $X$ with $\eta:=\Spec k(X)$. Then $W$ becomes a closed point of $Y_\eta$ and the correspondence $W_\eta:\eta\to Y$ factors as $p_2\circ i_{W_\eta}\circ p_1^t$, where $p_1:W_\eta\to \eta$  $p_2:Y_\eta\to Y$ are the projections.

Let $W'_\eta\subset Y_\eta\times  Z$ be the pull-back of $W'$. As we have already established naturality with respect to pull-back by flat maps, we reduce to showing
\[
(\P^1_?,1)(W'_\eta\circ  i_{W_\eta})=(\P^1_?,1)(W'_\eta)\circ (\P^1_?,1)( i_{W_\eta}).
\]
Since $Y$ is quasi-projective, we can find a sequence of closed subschemes of $Y_\eta$
\[
W_\eta=W_0\subset W_1\subset \ldots\subset W_{d-1}\subset W_d=Y_\eta
\]
such that $W_i$ is smooth of codimension $d-i$ on $Y_\eta$. Using again the fact the $\co$ is defined at the generic point, and that we have already proven functoriality with respect to composition of morphisms, we reduce to the case of $Y=\Spec \sO$ for some DVR $\sO$, and $i_\eta$ the inclusion of the closed point $\eta$ of $Y$. 

Let $W''\to W'$ be the normalization of $W'$. Using functoriality with respect to morphisms in $\Sm/k$ once more, we may replace $Z$ with $W''$ and $W'$ with the transpose of the graph of the projection $W''\to Y$. Changing notation, we may assume that $W'$ is the transpose of the graph of a finite morphism $Z\to Y$. This reduces us to the case considered in proposition~\ref{prop:LocalCyc}; this latter result completes the proof.
\end{proof}

We will collect the results of this section, generalized to higher loops, in theorem~\ref{thm:Main1} of the next section.

\section{Higher loops}
 The results of these last sections carry over immediately to statements about the  $n$-fold smash product $(\P^1,1)^{\wedge n}$ for $n\ge1$.  For clarity and completeness, we list these explicitly in an omnibus theorem.
 
Let $R$ be a semi-local $k$-algebra, smooth and essentially of finite type over $k$, and let  $\bar{x}\subset \P^1_R$ and $f$  be as in section~\ref{sec:cotrans}. For $n\ge1$, define
\[
\co^n_{\bar{x},\bar{f}}:\Sigma^n_{\P^1}\Spec R_+\to\Sigma^n_{\P^1}\bar{x}_+.
\]
be the map $\id_{(\P^1,1)^{\wedge n-1}}(\co_{\bar{x},\bar{f}})$. 

Similarly, let $A$ be a semi-local $k$-algebra, smooth and essentially of finite type over $k$. Let  $B=A[x]$ be a simply generated finite generically \'etale $A$-algebra. For $n\ge1$, define
\[
\co^n_{x}:\Sigma^n_{\P^1}\Spec A_+\to \Sigma^n_{\P^1}\Spec B_+.
\]
be the map $\Sigma^n_{\P^1}(\co_{x})$. 
 
\begin{thm} \label{thm:Main1} 1. For $\bar{x}=0$, $f=s$, we have $\co^n_{\bar{x},\bar{f}}=\id$.\\
\\
2. Let  $R\to R'$ be a flat extension of smooth semi-local  $k$-algebras, essentially of finite type over $k$.  Let $\bar{x}$ be a smooth  closed subscheme of $\P^1_R\setminus\{1\}$, finite and generically \'etale over $R$. Let $\bar{x}'=\bar{x}\times_RR'\subset \P^1_{R'}$. Let $\bar{f}$ be a generator for the $m_{\bar{x}}/m^2_{\bar{x}}$, and let $\bar{f}'$ be the extension to $m_{\bar{x}'}/m^2_{\bar{x}'}$. Then the diagram 
\[
\xymatrixcolsep{40pt}
\xymatrix{
\Sigma^n_{\P^1}\Spec R'_+\ar[r]^-{\co^n_{\bar{x}',\bar{f}'}}\ar[d]&\Sigma^n_{\P^1}\bar{x}'_+\ar[d]\\
\Sigma^n_{\P^1}\Spec R_+\ar[r]_-{\co^n_{\bar{x},\bar{f}}}&\Sigma^n_{\P^1}\bar{x}_+
}
\]
commutes.\\
\\
3. The co-group structure $\Sigma^{n-1}_{\P^1}(\sigma_{\P^1})$ on $(\P^1,1)^{\wedge n}$ is given by the map 
\[
\co^n_{\{0,\infty\},\overline{s/(s-1)^2}}:(\P^1,1)^{\wedge n}\to(\P^1,1)^{\wedge n}\vee (\P^1,1)^{\wedge n}.
\]
4. The co-group $((\P^1,1)^{\wedge n},\Sigma^{n-1}_{\P^1}(\sigma_{\P^1}))$ in $\SH_{S^1}(k)/f_{n+1}$ is co-commutative.\\
\\
5. For an extension $A\to B$ as above, the map $\co^n_{x}:\Sigma^n_{\P^1}\Spec A_+\to\Sigma^n_{\P^1}\Spec B_+$ is independent of the choice of $x$, and is denoted $\co^n_{B/A}$.\\
\\
6. Suppose that $\Char k=0$. The $\SH_{S^1}(k)/f_{n+1}$-valued presheaf on $\Sm/k^\op$
\[
\Sigma^n_{\P^1}?_+:\Sm/k\to \SH_{S^1}(k)/f_{n+1}
\]
extends  to an  $\SH_{S^1}(k)/f_{n+1}$-valued presheaf on $\SmCor(k)^\op$, by sending a generator $W\subset X\times Y$ of $\Hom_{\SmCor}(X,Y)$ to the morphism $\Sigma^n_{\P^1}X_+\to\Sigma^n_{\P^1}Y_+$ in $\SH_{S^1}(k)/f_{n+1}$ determined by the diagram
\[
\xymatrix{
\Sigma^n_{\P^1} \Spec k(X)_+\ar[r]^-\sim\ar[d]_{\co^n_{k(W)/k(X)}}&\Sigma^n_{\P^1}X_+\\
\Sigma^n_{\P^1}\Spec k(W)_+\ar[d]_{p_2}\\
\Sigma^n_{\P^1} Y_+
}
\]
in $\SH_{S^1}(k)/f_{n+1}$. The assertion that 
\[
\Sigma^n_{\P^1} \Spec k(X)_+\to\Sigma^n_{\P^1}X_+
\]
is an isomorphism in $\SH_{S^1}(k)/f_{n+1}$ is part of the statement. We write the map in $\SH_{S^1}(k)/f_{n+1}$ associated to $\alpha\in \Hom_{\SmCor}(X,Y)$ as
\[
\co^n(\alpha):\Sigma^n_{\P^1}X_+\to \Sigma^n_{\P^1}Y_+.
\]
\end{thm}

\section{Supports and co-transfers}\label{sec:Supp}

In this section, we assume that $\Char k=0$. We consider the following situation. Let $i:Y\to X$ be a codimension one closed immersion in $\Sm/k$, and let $Z\subset X$ be a pure codimension $n$ closed subset of $X$ such that $i^{-1}(Z)\subset Y$ also has pure codimension one. We let $T=i^{-1}(Z)$, $X^{(Z)}:=(X,X\setminus Z)$, $Y^T=(Y,Y\setminus T)$, so that $i$ induces the map of pointed spaces
\[
i:X^{(Z)}\to Y^{(T)}
\]
Let $z$ be the set of generic points of $Z$, $\sO_{X,z}$ the semi-local ring of $z$ in $X$, $X_z=\Spec\sO_{X,z}$ and $X_z^{(z)}=(X_z, X_z\setminus z)$. We let $t$ be the set of generic points of $T$, and let $\sO_{X,t}$ be the semi-local ring of $t$ in $X$, $X_t=\Spec\sO_{X,t}$. Set $Y_t:=X_t\times_XY$ and let  $Y^{(t)}_t=(Y_t, Y_t\setminus t)$. 

\begin{lem}\label{lem:ExcisionIso}  There are canonical isomorphisms in $\SH_{S^1}(k)/f_{n+1}$
\[
X^{(Z)}\cong X_z^{(z)}\cong  \Sigma_{\P^1}^n z_+;\quad Y^{(T)}\cong Y^{(t)}_t\cong \Sigma_{\P^1}^n t_+.
\]
\end{lem}

\begin{proof} This follows from lemma~\ref{lem:CanonThomIso}.
 \end{proof}
 
 Thus, the inclusion $i$ gives us the map in  $\SH_{S^1}(k)/f_{n+1}$:
 \[
 i:\Sigma_{\P^1}^n t_+\to \Sigma_{\P^1}^n z_+.
 \]
 On the other hand, we can define a map
 \[
  i_\co:\Sigma_{\P^1}^n t_+\to \Sigma_{\P^1}^n z_+
 \]
 as follows: Let $Z_t=Z\cap X_t\subset X_t$. Since $Y$ has codimension one in $X$ and intersects $Z$ properly, $t$ is a collection of codimenison one points of $Z$, and thus $Z_t$ is a semi-local reduced scheme of dimension one. Let $p:\tilde{Z}_t\to Z_t$ be the normalization, and let $\tilde{t}\subset \tilde{Z}_t$ be the set of points lying over $t\subset Z_t$. Write $\tilde{t}=\cup_j\tilde{t}_j$. For each $j$, we let $n_j$ denote the multiplicity at $\tilde{t}_j$ of the pull-back Cartier divisor $Y_t\times_{X_t}\tilde{Z}_t$, and let $t_j=p(\tilde{t}_j)$. This gives us the diagram
 \[
 \xymatrix{
 \tilde{t}\ar[r]^{\tilde{i}}\ar[d]_p& \tilde{Z}_t\ar[d]_p&z\ar[l]_j\\
 t\ar[r]_i&Z.
 }
 \]
Note that $j$ is an isomorphism in $\SH_{S^1}(k)/f_1$.  We define $i_\co$ to be the composition  
 \[
\Sigma_{\P^1}^nt_+\xrightarrow{\sum_j n_j\co_{\tilde{t}_j/t}^n}
\Sigma_{\P^1}^n\tilde{t}_+\xrightarrow{\Sigma_{\P^1}^n\tilde{i}}\Sigma_{\P^1}^n \tilde{Z}_+\xrightarrow{\Sigma_{\P^1}^n j^{-1}}\Sigma_{\P^1}^n z_+
 \]
  in $\SH_{S^1}(k)/f_{n+1}$.
  
\begin{lem}\label{lem:intersMult} $i=i_\co$ in  $\SH_{S^1}(k)/f_{n+1}$.
\end{lem}

\begin{proof} Using Nisnevich excision, we may replace $X$ with the Henselization of $X$ along $t$; we may also assume that $t$ is a single point. Via a limit argument, we may then replace $X$ with a smooth affine scheme of dimension $n+1$ over $k(t)$. Thus we may take $Z$ to be a reduced closed subscheme of $X$ of pure dimension one over $k(t)$.  We may also assume that $Y$ is the fiber over $0$ of a morphism $X\to\A^1_{k(t)}$ for which the restriction to $Z$ is finite.

As we are working in $\SH_{S^1}(k)/f_{n+1}$, we may replace $(X,Z)$ with $(X',Z')$ if there is a morphism $f:X\to X'$ which makes $(X,t)$ a Hensel neighborhood of $(X', f(t))$ and such that the restriction of $f$ to $Z'$ is birational.   Using Gabber's presentation lemma \cite[lemma 3.1]{Gabber}, we may assume that $X= \A^{n+1}_{k(t)}$, that $t=0$ and that $Y$ is the coordinate hyperplane $X_{n+1}=0$. We write $F$ for $k(t)$ and changing notation write simply 0 instead of $t$.

After a suitable linear change of coordinates in $\A^{n+1}_F$, we may assume that each coordinate projection 
\begin{align*}
&q:\A^{n+1}_F \to \A^r_F\\
&q(x_1,\ldots, x_{n+1})=(x_{i_1},\ldots, x_{i_r}),
\end{align*}
$r=1,\ldots, n$,  restricts to a finite morphism on  $Z$, and that $Z\to q(Z)$ is birational if $r\ge2$. 

We now reduce to the case in which $Z$ is contained in the coordinate subspace $X'=\A^2_F$ defined by $X_1=\ldots=X_{n-1}=0$. For this,  consider the map
\begin{align*}
m:\A^1\times\A^{n+1}_F&\to \A^1\times\A^{n+1}_F\\
m(t,x_1,\ldots, x_{n+1})&=(t, tx_1,\ldots, tx_{n-1},x_n, x_{n+1})
\end{align*}
Let $\sZ=m(\A^1\times Z)\subset \A^1\times\A^{n+1}_F$. By our finiteness assumptions, $\sZ$ is a (reduced) closed subscheme of $\A^1\times\A^{n+1}_F$, and each fiber $\sZ_t\subset t\times\A^{n+1}_F$ is birationally isomorphic to $Z\times_FF(t)$. Consider the inclusion map 
\[
(\A^1\times Y)^{(\A^1\times0)}\to (\A^1\times X)^{(\sZ)}
\]
The  maps
\[
i_0, i_1:Y^{(0)}\to (\A^1\times Y)^{(\A^1\times0)}
\]
are clearly isomorphisms in $\sH_\bullet(k)$, and the maps
\begin{align*}
&i_1:X^{(Z)}\to (\A^1\times X)^{(\sZ)}\\
&i_0:X^{(\sZ_0)}\to (\A^1\times X)^{(\sZ)}
\end{align*}
 are easily seen to be isomorphisms in $\SH_{S^1}(k)/f_{n+1}$. Combining this with the commutative diagram
\[
\xymatrix{
Y^{(0)}\ar[r]\ar[d]_{i_1}&X^{(Z)}\ar[d]^{i_1}\\
 (\A^1\times Y)^{(\A^1\times0)}\ar[r]&X^{(\sZ)}\\
Y^{(0)}\ar[r]\ar[u]^{i_0}&X^{(\sZ_0)}\ar[u]_{i_0}
}
\]
shows that we can replace $Z$ with $\sZ_0\subset X'$.

Having done this, we see that the map $Y^{(0)}\to X^{(Z)}$ is just the $n-1$-fold $\P^1$ suspension  of the map
\[
(Y\cap X')^{(0)}\to (X')^{(Z)}
\]
This reduces us to the case $n=1$.

Since $p_2:Z\to \A^1_F$ is finite, we may replace $\A^1\times\A^1_F$ with $\P^1\times\A^1_F$. Then the map $Y^{(0)}\to X^{(Z)}$ is isomorphic to $(\P^1\times0,\infty\times0)\to X^{(Z)}$. We extend this to the isomorphic map
\[
(\P^1\times\A^1_F,\infty\times\A^1_F)\to X^{(Z)}=(\P^1\times\A^1_F,\P^1\times\A^1_F\setminus Z).
\]
Let $s$ be the generic point of $\A^1_F$, $Z_s$ the fiber of $p_2$ over $s$. Then the inclusions
\begin{align*}
(\P^1\times0,\infty\times0)\xrightarrow{j_0}&(\P^1\times\A^1_F,\infty\times\A^1_F)\xleftarrow{j_s}(\P^1\times s,\infty\times s)\\
&(\P^1\times\A^1_F,\P^1\times\A^1_F\setminus Z)\xleftarrow{j_s}(\P^1\times s,\P^1_s\setminus Z_s)
\end{align*}
are isomorphisms in $\SH_{S^1}(k)/f_2$, and thus the map
\[
i_0:Y^{(0)}\cong (\P^1\times0,\infty\times0)\to X^{(Z)}=(\P^1\times\A^1_F,\P^1\times\A^1_F\setminus Z)
\]
is isomorphic in $\SH_{S^1}(k)/f_2$ to the collapse map
\[
(\P^1\times s,\infty\times s)\to (\P^1\times s,\P^1_s\setminus Z_s).
\]
Therefore, the map 
\[
 i:\Sigma_{\P^1} 0_+\to \Sigma_{\P^1} z_+
 \]
we need to consider is equal to the co-transfer map
\[
\co_{Z_s/s}:\Sigma_{\P^1} s_+\to \Sigma_{\P^1}z_{s+}
\]
composed with the (canonical) isomorphisms
\[
\Sigma_{\P^1}0_+\xrightarrow{i_0}\Sigma_{\P^1} s_+;\quad \Sigma_{\P^1}z_{s+}\cong \Sigma_{\P^1}z_+,
\]
the latter isomorphism arising by noting that $z_s$ is a generic point of $Z$ over $F$. The result now follows directly from proposition~\ref{prop:LocalCyc}.
\end{proof}

\begin{Def} \label{Def:CorSupp} 1. Take $X, X'\in\Sm/k$, and let  $Z\subset X$,   $Z'\subset X'$ be pure codimension $n$ closed subsets.  Take a generator $A\in\Hom_{\SmCor}(X,X')$,  $A\subset X\times X'$. Let $q:A^N\to A$ be the normalization of $A$. Let $z$ be the set of generic points of $Z$, let $a$ be the set of generic points of $A\cap X\times Z'$ and let $a'=q^{-1}(a)$.  Suppose that 
\begin{enumerate}
\item $A^N\to X$ is \'etale on a neighborhood of $a'$
\item $p_X(a)$ is contained in $Z$.  
\end{enumerate}
Let $\sO_{A^N,a}$ be the semi-local ring of $a'$ in $A^N$, and let $A^N_{a'}=\Spec \sO_{A^N,a'}$; define $X_z$ similarly. Define
\[
\co^n(W): X^{(Z)}\to X^{\prime(Z')}
\]
to be the map in $\SH_{S^1}(k)/f_{n+1}$ given by the following composition:
\[
X^{(Z)}\cong  X_z^{(z)}\cong  \Sigma^n_{\P^1}z_+\xrightarrow{\co^n_{a'/z}}
 \Sigma^n_{\P^1}a'_+ \cong A_{a'}^{N(a')}\xrightarrow{p_{X'}}X^{\prime(Z')}.
\]
2. Let $\Hom_{\SmCor}(X,X')_{Z, Z'}\subset \Hom_{\SmCor}(X,X')$ be the subgroup generated by $A$ satisfying (a) and (b). We extend the definition of the morphism $\co^n(A)$ to $\Hom_{\SmCor}(X,X')_{Z, Z'}$ by linearity.
\end{Def}
Note that we implicity invoke lemma~\ref{lem:ExcisionIso} to ensure that the isomorphisms used in the definition of $\co^n(A)$ exist and are canonical; condition (1) implies in particular that $A$ is smooth in a neighborhood of $a$, so we may use lemma~\ref{lem:ExcisionIso} for the isomorphism $\Sigma^n_{\P^1}a_+ \cong A_a^{(a)}$.

\begin{lem}\label{lem:CorSuppFunct}  Take $X, X', X''\in\Sm/k$, and let  $Z\subset X$,   $Z'\subset X'$ and $Z''\subset X''$be a pure codimension $n$ closed subsets. Take 
$\alpha\in \Hom_{\SmCor}(X,X')_{Z, Z'}$, $\alpha'\in \Hom_{\SmCor}(X',X'')_{Z', Z''}$. Then $\alpha'\circ\alpha$ is in $ \Hom_{\SmCor}(X,X'')_{Z, Z''}$ and
\[
\co^n(\alpha')\circ\co^n(\alpha)=\co^n(\alpha'\circ\alpha).
\]
\end{lem}

\begin{proof} We may assume that $\alpha$ and $\alpha'$ are generators $A$ and $A'$. We may replace $X, X'$ and $X''$ with the respective strict Henselizations along $z, z'$ and $z''$. Write $z=\{z_1,\ldots, z_r\}$, $z'=\{z'_1,\ldots, z'_s\}$, $z''=\{z''_1,\ldots, z''_t\}$. Then $A$ and $A'$ break up as a disjoint union of graphs of morphisms 
\[
f_{jk}:X_{z_k}\to X'_{z'_j};\quad g_{ij}:X'_{z'_j}\to X''_{z''_i}
\]
and $A'\circ A$ is thus the sum of the graphs of the compositions $g_{ij}\circ f_{jk}$. Therefore,  each irreducible component of the support of $A'\circ A$ is smooth. This verifies condition (1) of  definition~\ref{Def:CorSupp}; the condition (2) is easy and is left to the reader.

 The compatibility of $\co^n$ with the composition of correspondences follows directly from theorem~\ref{thm:Main1}(6).
\end{proof}

\begin{prop}\label{prop:BaseChange} Let $i:\Delta_1\to \Delta$ be a closed   immersion  of quasi-projective schemes in $\Sm/k$, take $X, X'\in \Sm/k$ and $\alpha\in \Hom_{\SmCor}(X,X')$. Let $Z\subset X\times\Delta$, $Z'\subset X'\times\Delta$ be closed codimension $n$ subsets. Suppose that
\begin{enumerate}
\item $Z_1:=Z\cap X\times\Delta_1$ and $Z_1':=Z'\cap X'\times\Delta_1$ have codimension $n$ in $X\times\Delta_1$, $X'\times\Delta_1$, respectively.
\item $\alpha\times\id_{\Delta}$ is in $\Hom_{\SmCor}(X\times\Delta,X'\times\Delta)_{Z, Z'}$
\item $\alpha\times\id_{\Delta_1}$ is in $\Hom_{\SmCor}(X\times\Delta_1,X'\times\Delta_1)_{Z_1, Z_1'}$
\end{enumerate}
Then the diagram in $\SH_{S^1}(k)/f_{n+1}$
\[
\xymatrixcolsep{50pt}
\xymatrix{
(X\times\Delta_1)^{(Z_1)}\ar[d]_{\id\times i}\ar[r]^-{\co^n(\alpha\times\id)}&(X'\times\Delta_1)^{(Z'_1)}\ar[d]^{\id\times i}\\
(X\times\Delta)^{(Z)}\ar[r]_-{\co^n(\alpha\times\id)}&(X'\times\Delta)^{(Z')}
}
\]
commutes.
\end{prop}

\begin{proof} Since $\Delta$ is by assumption quasi-projective, we may factor $\Delta_1\to \Delta$ as a sequence of closed codimension 1 immersions
\[
\Delta_1=\Delta^d\to\Delta^{d-1}\to\ldots\to \Delta^1\to\Delta^0=\Delta
\]
such that each closed immersion $\Delta^i\to \Delta$ satisfies the conditions of the proposition. This reduces us to the case of a codimension one closed immersion.

We may replace $X\times\Delta$, $X'\times\Delta$, etc., with the respective semi-local schemes about the generic points of $Z_1$ and $Z_1'$. As $\Delta_1$ has codimension one on $\Delta$, it follows that the normalizations $Z^N$, $Z^{\prime N}$ of $Z$ and $Z'$ are smooth over $k$. Let $\tilde{i}:\tilde{z}\to Z^N$, $\tilde{i}':\tilde{z}'\to Z^{\prime N}$ be the points of $Z^N$, $Z^{\prime N}$ lying over $Z_1, Z_1'$, respectively, which we write as a disjoint union of closed points
\[
\tilde{z}=\amalg_j\tilde{z}_j;\quad \tilde{z}'=\amalg_j\tilde{z}'_j.
\]
 By lemma~\ref{lem:ExcisionIso}  and lemma~\ref{lem:intersMult}, we may rewrite the diagram in the statement of the proposition as
\[
\xymatrixcolsep{50pt}
\xymatrix{
\Sigma_{\P^1}^nZ_{1+} \ar[d]_{\sum_jm_j\co^n_{\tilde{z}_j/Z_1}}\ar[r]^-{\co^n(\alpha\times\id_{Z_1^N})}&\Sigma_{\P^1}^n Z'_{1+}\ar[d]^{\sum_jm'_j\co^n_{\tilde{z}'_j/Z'_1}}\\
\Sigma_{\P^1}^n\tilde{z}_+\ar[d]_{\tilde{i}}&\Sigma_{\P^1}^n\tilde{z}_+'\ar[d]^{\tilde{i}'}\\
\Sigma_{\P^1}^nZ^N\ar[r]_-{\co^n(\alpha\times\id_{Z^N})}&\Sigma_{\P^1}^nZ^{\prime N}
}
\]
where $\alpha\times\id_{Z^N}$, $\alpha\times\id_{Z_1}$ denote the correspondences induced by $\alpha\times\id_{\Delta}$ and
$\alpha\times\id_{\Delta_1}$, and the $m_j, m_j'$ are the relevant intersection multiplicities. The commutativitiy of this diagram follows from the functoriality of the maps $\co^n_{-/-}(-)$ with respect to the composition of correspondences (theorem~\ref{thm:Main1}).
\end{proof}

\section{Slices of loop spectra} Take $E\in \SH_{S^1}(k)$. Following Voevodsky's remarks in \cite{VoevSlice}, Neeman's version of Brown representability \cite{Neeman} gives us the motivic Postnikov tower
\[
\ldots f_{n+1}E\to f_nE\to\ldots\to f_0E=E,
\]
where $f_nE\to E$ is universal for morphisms from an object of $\Sigma_{\P^1}^n\SH_{S^1}(k)$ to $E$. The layer $s_nE$ is the {\em $n$ slice} of $E$, and is characterized up to unique isomorphism by the distinguished triangle
\begin{equation}\label{eqn:DistTri}
f_{n+1}E\to f_nE\to s_nE\to \Sigma_sf_{n+1}E.
\end{equation}
The fact that this distinguished triangle determines $s_nE$ up to unique isomorphism rather than just up to isomorphism follows from
\begin{equation}\label{eqn:Vanishing}
\Hom_{\SH_{S^1}(k)}(\Sigma^{n+1}_{\P^1}\SH_{S^1}(k), s_nE)=0
\end{equation}
To see this, just use the universal property of $f_{n+1}E\to E$ and the long exact sequence of Homs associated to the distinguished triangle \eqref{eqn:DistTri}. In particular, using the description of $\Hom_{\SH_{S^1}(k)/f_{n+1}}(-,-)$ via right fractions we have

\begin{lem}\label{lem:Descent} For all $F,E\in \SH_{S^1}(k)$ and $n\ge0$,  the natural map
\[
\Hom_{\SH_{S^1}(k)}(F,s_nE)\to \Hom_{\SH_{S^1}(k)/f_{n+1}}(F,s_nE)
\]
is an isomorphism.
\end{lem}

See also \cite[proposition 5-3]{Verdier}

We recall the {\em de-looping formula} \cite[theorem 7.4.2]{LevineHC}
\[
s_n(\Omega_{\P^1}E)\cong \Omega_{\P^1}(s_{n+1}E)
\]
for $n\ge0$. 

Take $F\in \Spc_\bullet(k)$. For  $E\in \Spt_{S^1}(k)$, we have  
$\sHom^{int}(F,E)\in \SH$, which for $F=X_+$ is just $E(X)$, and in general is formed as the homotopy limit associated to the description of $F$ as a homotopy colimit of representable objects, i.e., take the Kan extension to $\Spc_\bullet$ of the functor $E:\Sm/k^\op\to \Spt_{S^1}(k)$. 

This gives us   the ``internal Hom" functor
\[
\sHom_{\SH_{S^1}(k)}(F,-):\SH_{S^1}(k)\to \SH_{S^1}(k)
\]
and more generally
\[
\sHom_{\SH_{S^1}(k)/f_{n+1}}(F,-):\SH_{S^1}(k)/f_{n+1}\to \SH_{S^1}(k),
\]
with natural transformation
\[
\sHom_{\SH_{S^1}(k)}(F,-)\to \sHom_{\SH_{S^1}(k)/f_{n+1}}(F,-).
\]
These have value on $E\in \Spt_{S^1}(k)$ defined by taking a fibrant model $\tilde{E}$ of $E$ (in $\SH_{S^1}(k)$ or $\SH_{S^1}(k)/f_{n+1}$, as the case may be) and forming the presheaf on $\Sm/k$
\[
X\mapsto \sHom^{int}(F\wedge X_+, \tilde{E}).
\]
Putting the de-looping formula together with lemma~\ref{lem:Descent} gives us

\begin{prop}\label{prop:Descent} For $E\in \SH_{S^1}(k)$  we have natural isomorphisms
\[
s_0(\Omega^n_{\P^1}E)\cong \Omega^n_{\P^1}s_nE\cong  \sHom_{\SH_{S^1}(k)/f_{n+1}}((\P^1,1)^{\wedge n}, s_nE)
\]
\end{prop}

\begin{proof}
Indeed,  the first isomorphism is just the de-looping isomorphism repeated $n$ times. For the second, we have
\begin{align*}
\Omega^n_{\P^1}s_nE&\cong \sHom_{\SH_{S^1}(k)}((\P^1,1)^{\wedge n},s_nE)\\
&\cong \sHom_{\SH_{S^1}(k)/f_{n+1}}((\P^1,1)^{\wedge n}, s_nE)
\end{align*}
the second isomorphism following from lemma~\ref{lem:Descent}.
\end{proof}

\begin{Def}\label{Def:Tranfer} Suppose that $\Char k=0$. Take $E\in \SH_{S^1}(k)$ and take $\alpha\in \Hom_{\SmCor}(X,Y)$. Define the {\em transfer}
\[
\Tr_{Y/X}(\alpha): (\Omega^n_{\P^1}s_nE)(Y)\to (\Omega^n_{\P^1}s_nE)(X)
\]
as follows: 
\begin{align*}
(\Omega^n_{\P^1}s_nE)(Y)&\cong \sHom^{int}{\SH_{S^1}(k)}(Y_+, \Omega^n_{\P^1}s_nE)\\
&\cong  \sHom^{int}_{\SH_{S^1}(k)}(\Sigma^n_{\P^1}Y_+,  s_nE)\\
&\cong  \sHom^{int}_{\SH_{S^1}(k)/f_{n+1}}(\Sigma^n_{\P^1}Y_+,  s_nE)\\
&\xrightarrow{\co^n(\alpha)^*}\sHom^{int}_{\SH_{S^1}(k)/f_{n+1}}(\Sigma^n_{\P^1}X_+,  s_nE)\\
&\cong  \sHom^{int}_{\SH_{S^1}(k)}(X_+, \Omega^n_{\P^1}s_nE)\\
&(\Omega^n_{\P^1}s_nE)(X).
\end{align*}
\end{Def}

\begin{thm}\label{thm:Transfer} Suppose that $\Char k=0$. For $E\in \SH_{S^1}(k)$, the maps $\Tr(\alpha)$ extend the presheaf
\[
\Omega^n_{\P^1}s_n E:\Sm/k^\op\to \SH
\]
to an $\SH$-valued presheaf with transfers
\[
\Omega^n_{\P^1}s_n E:\SmCor(k)^\op\to \SH
\]
\end{thm}

\begin{proof} This follows from the definition of the maps $\Tr(\alpha)$ and theorem~\ref{thm:Main1}, the main point being that the maps 
$\Tr(\alpha)$ factor through an internal Hom in $\SH_{S^1}(k)/f_{n+1}$.
\end{proof}

\begin{cor}\label{cor:Trans0}  Suppose that $\Char k=0$. For $E\in \SH_{S^1}(k)$, there is an extension of the presheaf
\[
s_0\Omega_{\P^1} E:\Sm/k^\op\to \SH
\]
to an $\SH$-valued presheaf with transfers
\[
s_0\Omega_{\P^1} E:\SmCor(k)^\op\to \SH.
\]
\end{cor}

\begin{proof} This is just the case $n=1$ of theorem~\ref{thm:Transfer}, together with the de-looping isomorphism
\[
s_0\Omega_{\P^1} E\cong  \Omega_{\P^1} s_1E.
\]
\end{proof}

\begin{rem} The corollary is actually the main result, in that one can deduce  theorem~\ref{thm:Transfer} from corollary~\ref{cor:Trans0} (applied to  $\Omega^{n-1}_{\P^1}E$) and the de-looping formula
\[
\Omega^n_{\P^1}s_nE\cong s_0\Omega^n_{\P^1}E=s_0\Omega_{\P^1}(\Omega^{n-1}_{\P^1}E).
\]
As the maps $\co^n(\alpha)$ are defined by smashing $\co^1(\alpha)$ with an identity map, this procedure does indeed give back the maps
\[
\Tr(\alpha):\Omega^n_{\P^1}s_nE(Y)\to \Omega^n_{\P^1}s_nE(X)
\]
as defined above.
\end{rem}

\begin{proof}[proof of theorem~\ref{IntroThm:Loops}] The weak transfers defined above give rise to homotopy invariant sheaves with transfers in the usual sense by taking the sheaves of homotopy groups of the motivic spectrum in question. For instance, corollary~\ref{cor:Trans0}  gives the   sheaf $\pi_m(s_0\Omega_{\P^1}E)$ the structure of a
homotopy invariant sheaf with transfers, in particular, an effective motive. In fact, these are {\em birational motives} in the sense of Kahn-Huber-Sujatha \cite{HuberKahn, KahnSujatha}, as $s_0F$ is a birational $S^1$-spectrum for each $S^1$-spectrum $F$. The classical Postnikov tower thus gives us a spectral sequence
\[
E^2_{p,q}:=H^{-p}(X_\Nis, \pi_q(s_0\Omega_{\P^1}E))\Longrightarrow \pi_{p+q}(s_0\Omega_{\P^1}E(X))
\]
with $E^2$ term a ``generalized motivic cohomology" of $X$. As the sheaves $\pi_q(s_0\Omega_{\P^1}E)$ are motives, we may replace Nisnevich cohomology with Zariski cohomology; as the sheaves $\pi_q(s_0\Omega_{\P^1}E)$ are birational, i.e., Zariski locally trivial, the higher Zariski cohomology vanishes, giving us
\[
\pi_n(s_0\Omega_{\P^1}E(X))\cong H^0(X_\Zar, \pi_n(s_0\Omega_{\P^1}E))= \pi_n(s_0\Omega_{\P^1}E(k(X)).
\]
\end{proof}

In short, we have  shown that the 0th slice of a $\P^1$-loop spectrum has transfers in the weak sense. We have already seen in section~\ref{sec:Example} that this does not hold for an arbitrary object of $\SH_{S^1}(k)$; in the next section we will see that the higher slices of an arbitrary $S^1$-spectrum do have transfers, albeit in an even weaker sense than the one used above.

\section{Transfers on the generalized cycle complex} \label{sec:GenCyc} We begin by recalling from \cite[theorem 7.1.1]{LevineHC} models for $f_nE$ and $s_nE(X)$ that are reminiscent of Bloch's higher cycle complex \cite{AlgCyc}. To simplify the notation, we will always assume that we have taken a model $E\in\Spt_{S^1}(k)$ which is quasi-fibrant.

For a scheme $X$ of finite type and locally equi-dimensional over $k$, let $\sS^{(n)}_X(m)$ be the set of closed subsets $W$ of $X\times\Delta^m$  of codimension $\ge n$, such that, for each face $F$ of $\Delta^n$, $W\cap X\times F$ has codimension $\ge n$ on $X\times F$ (or is empty). We order by $\sS^{(n)}_X(m)$  inclusion. 

For $X\in \Sm/k$, we let
\[
E^{(n)}(X,m):=\colim_{W\in \sS^{(n)}_X(m)}E^{(W)}(X\times\Delta^m),
\]
where $E^{(W)}(X)$ is by definition the homotopy fiber of the restriction map $E(X\times\Delta^n)\to E(X\times\Delta^n\setminus W)$. Similarly, for $0\le n\le n'$, we define
\[
E^{(n/n')}(X,m):=\colim_{W\in \sS^{(n)}_X(m), W'\in  \sS^{(n')}_X(m) }E^{(W\setminus W')}(X\times\Delta^m\setminus W')
\]

The conditions on the intersections of $W$ with $X\times F$ for faces $F$ means that   $m\mapsto \sS^{(n)}_X(m)$ form a cosimplicial set, denoted $\sS^{(n)}_X$, for each $n$ and that $\sS^{(n')}_X$ is a cosimplicial subset of $\sS^{(n)}_X$ for $n\le n'$. Thus the restriction maps for $E$ make $m\mapsto E^{(n)}(X,m)$ and $m\mapsto E^{(n/n')}(X,m)$ simplicial spectra, denoted $E^{(n)}(X,-)$ and $E^{(n/n')}(X,-)$. We denote the associated total spectra by $|E^{(n)}(X,-)|$ and $|E^{(n/n')}(X,-)|$.

 The inclusion $\sS^{(n')}_X(m)\to \sS^{(n)}_X(m)$ for $n\le n'$ and the evident restriction maps give the sequence
\[
|E^{(n')}(X,-)|\to |E^{(n)}(X,-)|\to |E^{(n/n')}(X,-)|
\]
which is easily seen to be a weak homotopy fiber sequence.

We note that $|E^{(0)}(X,-)|=E(X\times\Delta^*)$; as $E$ is homotopy invariant, the canonical map
\[
E(X)\to |E^{(0)}(X,-)|
\]
is thus a weak equivalence. We therefore have the tower in $\SH$
\begin{equation}\label{eqn:HCTower}
\ldots\to |E^{(n+1)}(X,-)|\to |E^{(n)}(X,-)|\to \ldots\to |E^{(0)}(X,-)|\cong E(X)
\end{equation}
with $n$th layer isomorphic to $|E^{(n/n+1)}(X,-)|$. We call this tower the {\em homotopy coniveau tower} for $E(X)$. In this regard, one of the main results from \cite{LevineHC} states

\begin{thm}[\hbox{\cite[theorem 7.1.1]{LevineHC}}]\label{thm:Slice} There is a canonical isomorphism of the tower \eqref{eqn:HCTower} with the motivic Postnikov tower evaluated at $X$:
\[
\ldots\to f_{n+1}E(X)\to f_nE(X)\to \ldots\to f_0E(X)=E(X),
\]
giving a canonical isomorphism
\[
s_nE(X)\cong |E^{(n/n+1)}(X,-)|.
\]
\end{thm}

We can further modify this description of $s_nE(X)$ as follows: Since $s_n$ is an idempotent functor, we have
\[
s_nE(X)\cong s_n(s_nE)(X)\cong  |(s_nE)^{(n/n+1)}(X,-)|
\]
Note that $ |(s_nE)^{(n/n+1)}(X,-)|$ fits into a weak homotopy fiber sequence
\[
 |(s_nE)^{(n+1)}(X,-)|\to  |(s_nE)^{(n)}(X,-)|\to  |(s_nE)^{(n/n+1)}(X,-)|
 \]
Using theorem~\ref{thm:Slice} in reverse, we have  the isomorphism in $\SH$
\[
 |(s_nE)^{(n+1)}(X,-)|\cong f_{n+1}(s_nE)(X)
\]
But as $f_{n+1}\circ f_n\cong f_{n+1}$, we see that $f_{n+1}(s_nE)\cong 0$ in $\SH_{S^1}(k)$ and thus 
\[
|(s_nE)^{(n)}(X,-)|\cong  |(s_nE)^{(n/n+1)}(X,-)|\cong s_nE(X)
\]
We may therefore use the simplicial model $|(s_nE)^{(n)}(X,-)|$ for $s_nE(X)$.

We will need a refinement of this construction, which takes into account the interaction of the support conditions with a given correspondence. 

\begin{Def} Let $A\subset Y\times X$ be a generator in $\Hom_{\SmCor}(Y,X)$; for each $m$, we let $A(m)\in \Hom_{\SmCor}(Y\times\Delta^m,X\times\Delta^m)$ denote the correspondence $A\times\id_{\Delta^m}$. Let $\sS^{(n)}_{X,A}(m)$ be the subset of $\sS^{(n)}_X(m)$ consisting of those $W'\in  \sS^{(n)}_X(m)$ such that
\begin{enumerate}
\item $W:=p_{Y\times\Delta^m}(A\times\Delta^m\cap Y\times W')$ is in $\sS^{(n)}_Y(m)$.
\item $A(m)$ is in $\Hom_{\SmCor}(Y\times\Delta^m,X\times\Delta^m)_{W,W'}$.
\end{enumerate}
For an arbitrary $\alpha\in \Hom_{\SmCor}(Y,X)$, write
\[
\alpha=\sum_{i=1}^r n_iA_i
\]
with the $A_i$ generators and the $n_i$ non-zero integers and define
\[
\sS^{(n)}_{X,\alpha}(m):=\cap_{i=1}^r\sS^{(n)}_{X,A_i}(m).
\]
If we have in addition to $\alpha$ a finite correspondence $\beta \in \Hom_{\SmCor}(Z,Y)$, we let $\sS^{(n)}_{X,\alpha,\beta}(m)\subset \sS^{(n)}_{X,\alpha}(m)$ be the set of $W\subset X\times\Delta^m$ such that $W$ is in $\sS^{(n)}_{X,\alpha}(m)$ and $p_{Y\times\Delta^m}(Y\times W\cap|\alpha|\times\Delta^m)$ is in $\sS^{(n)}_{Y,\beta}(m)$.
\end{Def}

For $f:Y\to X$ a flat morphism, one has
\[
\sS^{(n)}_{X,\Gamma_f}(m)=\sS^{(n)}_{X}(m)
\]
and for $g:Z\to Y$ a flat morphism, and $\alpha$ arbitrary, one has
\[
\sS^{(n)}_{X,\alpha, \Gamma_f}(m)=\sS^{(n)}_{X,\alpha}(m)
\]

Note that $m\mapsto \sS^{(n)}_{X,\alpha}(m)$ and $m\mapsto \sS^{(n)}_{X,\alpha,\beta}(m)$ define cosimplicial subsets of 
$m\mapsto \sS^{(n)}_{X}(m)$. We define the simplicial spectra  $E^{(n)}(X,-)_\alpha$ and $E^{(n)}(X,-)_{\alpha,\beta}$ using the support conditions $\sS^{(n)}_{X,\alpha}(m)$ and 
$\sS^{(n)}_{X,\alpha,\beta}(m)$ instead of $\sS^{(n)}_{X}(m)$:
\begin{align*}
&E^{(n)}(X,m)_\alpha:=\colim_{W\in \sS^{(n)}_{X,\alpha}(m)}E^{(W)}(X\times\Delta^m)\\
&E^{(n)}(X,m)_{\alpha,\beta}:=\colim_{W\in \sS^{(n)}_{X,\alpha,\beta}(m)}E^{(W)}(X\times\Delta^m)
\end{align*}
giving us the sequence of simplicial spectra
\[
E^{(n)}(X,-)_{\alpha,\beta}\to E^{(n)}(X,-)_{\alpha}\to E^{(n)}(X,-).
\]

The main ``moving lemma" \cite[theorem 2.6.2(2)]{LevineML}  yields
\begin{prop} \label{prop:ChowML} For $X\in\Sm/k$ affine, and $E\in \Spt_{S^1}(k)$ quasi-fibrant, the 
maps
\[
|E^{(n)}(X,-)_{\alpha,\beta}|\to |E^{(n)}(X,-)_{\alpha}|\to |E^{(n)}(X,-)|
\]
are weak equivalences.
\end{prop}

We proceed to the main construction of this section. Consider the simplicial model $|(s_nE)^{(n)}(X,-)|$ for $s_nE(X)$. For each $m$, we may consider the classical Postnikov tower for the spectrum $(s_nE)^{(n)}(X,m)$, which we write as
\[
\ldots\to \tau_{\ge p+1}(s_nE)^{(n)}(X,m)\to \tau_{\ge p}(s_nE)^{(n)}(X,m)\to\ldots\to (s_nE)^{(n)}(X,m),
\]
where 
\[
\tau_{\ge p+1}(s_nE)^{(n)}(X,m)\to (s_nE)^{(n)}(X,m)
\]
is the $p$-connected cover of $(s_nE)^{(n)}(X,m)$. The $p$th layer in this tower is of course the Eilenberg-Maclane spectrum on $\pi_p((s_nE)^{(n)}(X,m))$, or rather its $p$th suspension. Taking a functorial model for the 
 $p$-connected cover, we have for each $p$ the simpicial spectrum
 \[
 m\mapsto \tau_{\ge p+1}(s_nE)^{(n)}(X,m)
 \]
 giving us the tower of total spectra
\begin{equation}\label{eqn:SimplPostnikov}
\ldots\to |\tau_{\ge p+1}(s_nE)^{(n)}(X,-)|\to |\tau_{\ge p}(s_nE)^{(n)}(X,-)|\to\ldots\to |(s_nE)^{(n)}(X,-)|.
\end{equation}
The layers in this tower are then (up to suspension) the Eilenberg-Maclane spectrum on the chain complex $\pi_p(s_nE)^{(n)}(X,*)$, with differential as usual the alternating sum of the face maps. 

The chain complexes $\pi_p(s_nE)^{(n)}(X,*)$ are evidently functorial for smooth maps and inherit the homotopy invariance property from 
$(s_nE)^{(n)}(X,*)$ (see \cite[theorem 3.3.5]{LevineML}). Somewhat more surprising is

\begin{lem}\label{lem:NisExc} The complexes $\pi_p(s_nE)^{(n)}(X,*)$ satisfy Nisnevich excision.
\end{lem}

\begin{proof} Let $W\subset X\times\Delta^m$ be a closed subset in $\sS^{(n)}_X(m)$, and let $w$ be the set of generic points of $W$. Then
\[
s_nE^{(W)}(X\times\Delta^m)\cong s_nE(\Sigma^n_{\P^1}w_+)\cong\Omega^n_{\P^1}(s_nE)(w)\cong s_0(\Omega^n_{\P^1}E)(w).
\]
This gives us the following description of $\pi_p((s_nE)^{(n)}(X,m))$:
\[
(s_nE)^{(n)}(X,m)\cong\oplus_w \pi_p(s_0(\Omega^n_{\P^1}E)(w))
\]
where the direct sum is over the set $\sT^{(n)}_X(m)$ of  generic points of irreducible $W\in \sS^{(n)}_X(m)$. 

Now let $i:Z\to X$ be a closed subset with open complement $j:U\to X$. For each $m$, we thus have the exact sequence
\begin{multline*}
0\to \oplus_{w\in Z\times\Delta^m\cap \sT^{(n)}_X(m)} \pi_p(s_0(\Omega^n_{\P^1}E)(w))\to
\oplus_{w\in \sT^{(n)}_X(m)} \pi_p(s_0(\Omega^n_{\P^1}E)(w))\\\to
\oplus_{w\in \sT^{(n)}_X(m)\cap U\times\Delta^m} \pi_p(s_0(\Omega^n_{\P^1}E)(w))\to 0
\end{multline*}
Define the subcomplex $\pi_p(s_nE)^{(n)}(X,*)_Z$ of $\pi_p(s_nE)^{(n)}(X,*)$ and quotient complex $\pi_p(s_nE)^{(n)}(U_X,*)$ of 
$\pi_p(s_nE)^{(n)}(X,*)$ by taking supports in 
\[
W\in \sS^{(n)}_X(m)\cap Z\times\Delta^*,\ W\in \sS^{(n)}_X(m)\cap U\times\Delta^*. 
\]
We thus have the term-wise exact sequence of complexes
\[
0\to \pi_p(s_nE)^{(n)}(X,*)_Z\to \pi_p(s_nE)^{(n)}(X,*)\to \pi_p(s_nE)^{(n)}(U_X,*)\to0
\]

The localization technique of  \cite[theorem 8.10]{LevineTL} (for details, see \cite[theorem 3.2.1]{LevineHC}) implies that the inclusion 
\[
\pi_p(s_nE)^{(n)}(U_X,*)\to \pi_p(s_nE)^{(n)}(U,*)
\]
is a quasi-isomorphism, and we therefore have the quasi-isomorphism
\[
\pi_p(s_nE)^{(n)}(X,*)_Z\to \cone( \pi_p(s_nE)^{(n)}(X,*)\xrightarrow{j^*} \pi_p(s_nE)^{(n)}(U,*)[-1].
\]
But the left-hand side only depends on the the Nisnevich neighborhood of $Z$ in $X$, which yields the desired Nisnevich excision property.
\end{proof}

We will use the results of section~\ref{sec:Supp} to give  $X\mapsto \pi_p(s_nE)^{(n)}(X,*)$ the structure of a complex of homotopy invariant  presheaves with transfer on $\Sm/k$, i.e. a motive.

For this, we consider the  complexes $\pi_p(s_nE)^{(n)}(X,*)_\alpha$, $\pi_p(s_nE)^{(n)}(X,*)_{\alpha,\beta}$ constructed above. The refined support condition are constructed so that, for each $W\in  \sS^{(n)}_{X,\alpha}(m)$, $\alpha$ is in $\Hom_{\SmCor}(Y,X)_{W',W}$, where 
\[
W'=p_1(Y\times\Delta^m\times W\cap|\alpha|\times\Delta^m). 
\]
We may therefore use the morphism $\co^n(\alpha\times\id_{\Delta^m})$ to define the map
\[
\Tr_{Y/X}(\alpha)(m):\pi_p((s_nE)^{(n)}(X,m))_\alpha\to \pi_p(s_nE)^{(n)}(Y,m).
\]
By proposition~\ref{prop:BaseChange}, the maps $\Tr_{Y/X}(m)$ define a map of complexes
\[
\Tr_{Y/X}(\alpha):\pi_p(s_nE)^{(n)}(X,*)_\alpha\to \pi_p(s_nE)^{(n)}(Y,*).
\]

Similarly, given $\beta\in\Hom_{\SmCor}(Z,Y)$, we have the map of complexes
\[
\Tr_{Y/X}(\alpha)_\beta:\pi_p(s_nE)^{(n)}(X,*)_{\alpha,\beta}\to \pi_p(s_nE)^{(n)}(Y,*)_\beta.
\]
Note that, due to possible ``cancellations" occurring when one takes the composition $\alpha\circ\beta$, we have only an inclusion
\[
\sS^{(n)}_{X,\alpha,\beta}(m)\subset \sS^{(n)}_{X,\alpha\circ\beta}(m)
\]
giving us a natural comparison map
\[
\iota_{\alpha,\beta}:\pi_p(s_nE)^{(n)}(X,*)_{\alpha,\beta}\to :\pi_p(s_nE)^{(n)}(X,*)_{\alpha\circ\beta}.
\]
Using our moving lemma again, we see that $\iota_{\alpha,\beta}$ is a quasi-isomorphism in case $X$ is affine.

\begin{lem} Suppose $\Char k=0$. For  
\[
\alpha\in \Hom_{\SmCor}(Z,Y),\ \beta\in\Hom_{\SmCor}(Z,Y), 
\]
we have
\[
\Tr_{Z/Y}(\beta)\circ \Tr_{Y/X}(\alpha)_\beta=\Tr_{Z/X}(\alpha\circ\beta)\circ \iota_{\alpha,\beta}.
\]
\end{lem}

\begin{proof} This follows from lemma~\ref{lem:CorSuppFunct}.
\end{proof}

We have already noted that complexes $\pi_p(s_nE)^{(n)}(X,*)$  are functorial in $X$ for flat morphisms in $\Sm/k$, in particular for smooth  morphisms in $\Sm/k$. Let $\widetilde{\Sm}/k$ denote the subcategory of $\Sm/k$ with the same  objects and with morphisms the  smooth morphisms. The transfer maps we have defined on the refined complexes, together with the moving lemma~\ref{lem:CorSuppFunct} yield the following result:
\begin{thm}\label{thm:CorSuppExt} Suppose $\Char k=0$. Consider the presheaf
\[
\pi_p((s_nE)^{(n)}(-,*)):\widetilde{\Sm}/k^\op\to C^-(\Ab)
\]
on $\widetilde{\Sm}/k^\op$. Let 
\[
\iota:\widetilde{\Sm}/k\to SmCor(k)
\]
be the evident inclusion and let 
\[
Q:C^-(\Ab)\to D^-(\Ab)
\]
be the evident additive functor. There is a complex of presheaves with transfers
\[
\hat{\pi}_p((s_nE)^{(n)})^*:\SmCor(k)^\op\to C^-(\Ab)
\]
and an isomorphism of functors from $\widetilde{\Sm}/k^\op$ to $D^-(\Ab)$
\[
Q\circ \pi_p((s_nE)^{(n)}(-,*))\cong Q\circ \hat{\pi}_p((s_nE)^{(n)})^*\circ \iota.
\]
\end{thm}

\begin{proof} We give a rough sketch of the construction here; for details we refer the reader to  \cite[proposition 2.2.3]{KahnLevine}, which in turn is an elaboration of \cite[theorem 7.4.1]{LevineML}.  The construction of $\hat{\pi}_p((s_nE)^{(n)})^*$ is accomplished by first taking a homotopy limit over the complexes $\pi_p(s_nE)^{(n)}(X,*)_\alpha$. These are then functorial on $\SmCor(k)^\op$, up to homotopy equivalences arising from the replacement of the index category for the homotopy limit with a certain cofinal subcategory. One then forms a regularizing homotopy colimit that is strictly functorial on $\SmCor(k)^\op$, and finally, one replaces this presheaf with a fibrant model. The moving lemma for affine schemes (proposition~\ref{prop:ChowML}) implies that the homotopy limit construction yields for each affine $X\in\Sm/k$ a complex canonically quasi-isomorphic to $\pi_p(s_nE)^{(n)}(X,*)$; this property is inherited by the regularized homotopy colimit. As the complexes $\pi_p(s_nE)^{(n)}(X,*)$ satisfy Nisnevich excision (lemma~\ref{lem:NisExc}) and are homotopy invariant for {\em all X}, this implies that the fibrant model $\hat{\pi}_p((s_nE)^{(n)})^*$ is canonically quasi-isomorphic to  $\pi_p(s_nE)^{(n)}(X,*)$ for all $X\in\Sm/k$. 
\end{proof}

\begin{cor} \label{cor:CorSuppExt}  Suppose $\Char k=0$. $\hat{\pi}_p((s_nE)^{(n)})^*$ is a homotopy invariant complex of presheaves with transfer.
\end{cor}

\begin{proof}  By theorem~\ref{thm:CorSuppExt}, we have the  isomorphism in $D^-(\Ab)$
\[
\hat{\pi}_p((s_nE)^{(n)})^*\cong  \pi_p((s_nE)^{(n)}(-,*)).
\]
for all $X\in \Sm/k$. As  the presheaf $\pi_p((s_nE)^{(n)}(-,*))$ is homotopy invariant, so is  $\hat{\pi}_p((s_nE)^{(n)})^*$.
\end{proof}

\begin{proof}[proof of theorem~\ref{IntroThm:Tower}] As in the proof of theorem~\ref{thm:CorSuppExt}, the method of \cite[theorem 7.4.1]{LevineML}, shows that the tower  \eqref{eqn:SimplPostnikov} extends to a tower 
\begin{equation}\label{eqn:FunctPostnikov} 
\ldots\to\rho_{\ge p+1}s_nE\to \rho_{\ge p}s_nE\to \ldots\to s_nE
\end{equation}
in $\SH_{S^1}(k)$ with value \eqref{eqn:SimplPostnikov} at $X\in \Sm/k$, and with the cofiber of $\rho_{\ge p+1}s_nE\to \rho_{\ge p}s_nE$ naturally isomorphic to $\EM^\eff(\hat{\pi}_p((s_nE)^{(n)})^*)$. By corollary~\ref{cor:CorSuppExt},   the pre\-sheaves $\hat{\pi}_p((s_nE)^{(n)})^*$ define objects in $\DM^\eff_-(k)$. Thus,  we have shown that the layers in the tower \eqref{eqn:FunctPostnikov}  have a ``motivic" structure, proving theorem~\ref{IntroThm:Tower}.
\end{proof}

\section{The Friedlander-Suslin tower} As the reader has surely noticed, the lack of functoriality for the simplicial spectra $E^{(n)}(X,-)$ creates annoying technical problems when we wish to extend the construction of the homotopy coniveau tower to a tower in $\SH_{S^1}(k)$. In their work on the spectral sequence from motivic cohomology to $K$-theory, Friedlander and Suslin \cite{FriedSus} have constructed a completely functorial version of the homotopy coniveau tower, using ``quasi-finite supports".  Unfortunately, the comparison between the Fried\-lander-Suslin version and $E^{(n)}(X,-)$ is proven in \cite{FriedSus} only for $K$-theory and motivic cohomology. In this last section, we recall the Fried\-lander-Suslin construction and form the conjecture that the  Friedlander-Suslin tower is naturally isomorphic to the homotopy coniveau tower. 

Let $\sQ_X^{(n)}(m)$ be the set of closed subsets $W$ of $\A^n\times X\times\Delta^m$ such that, for each irreducible component $W'$ of $W$, the projection  $W'\to X\times\Delta^m$ is quasi-finite. For $E\in\Spt_{S^1}(k)$, we let 
\[
E^{(n)}_{FS}(X,m):=\colim_{W\in \sQ^{(n)}_X(m)}E^{(W)}(\A^n\times X\times\Delta^m)
\]
As the condition defining $\sQ_X^{(n)}(m)$ are preserved under maps 
\[
\id_{\A^n}\times f\times g:\A^n\times X'\times\Delta^{m'}\to \A^n\times X\times\Delta^m, 
\]
where $f:X'\to X$ is an arbitrary map in $\Sm/k$, and $g:\Delta^{m'}\to \Delta^m$ is a structure map in $\Delta^*$, the spectra $E^{(n)}_{FS}(X,m)$ define a simplicial spectrum $E^{(n)}_{FS}(X,-)$ and these simplicial spectra, for $X\in \Sm/k$, extend to a presheaf of simplicial spectra on $\Sm/k$:
\[
E^{(n)}_{FS}(?,-):\Sm/k^\op\to \Delta^\op\Spt.
\]
Similarly, if we take the linear embedding $i_n:\A^n\to \A^{n+1}=\A^n\times\A^1$, $x\mapsto (x,0)$, the  pull-back by $i_n\times\id$ preserves the support conditions, and thus gives a well-defined map of simplicial spectra
\[
i_n^*:E^{(n+1)}_{FS}(X,-)\to E^{(n)}_{FS}(X,-),
\]
forming the tower of presheaves on $\Sm/k$
\begin{equation}\label{eqn:FSTower}
\ldots\to E^{(n+1)}_{FS}(?,-)\to E^{(n)}_{FS}(?,-)\to\ldots
\end{equation}

We may compare $E^{(n)}_{FS}(X,-)$ and $E^{(n)}(X,-)$ using the method of \cite{FriedSus} as follows: The simplicial spectra $E^{(n)}(X,-)$ are functorial for flat maps in $\Sm/k$, in the evident manner. They satisfy homotopy invariance, in that the pull-back map
\[
p^*:E^{(n)}(X,-)\to E^{(n)}(\A^1\times X,-)
\]
induces a weak equivalence on the total spectra. We have the evident inclusion of simplicial sets
\[
\sFS_X^{(n)}(-)\hookrightarrow \sS^{(n)}_{\A^n\times X}(-)
\]
inducing the map
\[
\phi_{X,n}:E^{(n)}_{FS}(X,-)\to E^{(n)}(\A^n\times X,-) 
\]

Together with the weak equivalence  $p^*:|E^{(n)}(X,-)| \to |E^{(n)}(\A^n\times X,-)|$, the maps $\phi_{X,n}$ induce a map of towers of total spectra in $\SH$

\begin{equation}\label{eqn:FSComp}
\phi_{X,*}: |E^{(*)}_{FS}(X,-)|\to |E^{(*)}(X,-)|.
\end{equation}

\begin{conj} \label{conj:WH1} For each $X\in\Sm/k$ and each quasi-fibrant $E\in \Spt_{S^1}(k)$, the map \eqref{eqn:FSComp} induces an isomorphism in $\SH$ of the towers of total spectra.
\end{conj}

Combined with the weak equivalence given by homotopy invariance and the results of \cite{LevineHC}, this would give us an isomorphism in $\SH_{S^1}(k)$:
\[
f_nE\cong |E^{(n)}_{FS}(?,-)|
\]

\end{document}